\documentclass[11pt]{amsart}
\usepackage[utf8]{inputenc}
\usepackage{amssymb,amsmath}
\usepackage{amsthm}
\usepackage{dsfont}
\usepackage{graphicx}
\usepackage{color}
\usepackage{tikz}
\usepackage{pgfplots}
\usepackage{enumitem}
\usepackage{setspace}
\usepackage{mathtools}
\usepackage{xfrac}

\usepackage{multicol} 
\usepackage{soul}
\usepackage{subcaption}
\usepackage{booktabs}
\usepackage{array}
\newcolumntype{L}[1]{>{\raggedright\let\newline\\\arraybackslash\hspace{0pt}}m{#1}}
\newcolumntype{C}[1]{>{\centering\let\newline\\\arraybackslash\hspace{0pt}}m{#1}}
\newcolumntype{R}[1]{>{\raggedleft\let\newline\\\arraybackslash\hspace{0pt}}m{#1}}
\usepackage[linesnumbered, ruled,vlined]{algorithm2e}

\usepackage[hidelinks]{hyperref}
\DeclareMathOperator*{\esssup}{ess\,sup}
\DeclareMathOperator*{\essinf}{ess\,inf}

\mathtoolsset{showonlyrefs}

\newcommand\Mtilde{\stackrel{\, \sim}{\smash{M}\rule{0pt}{1.2ex}}}
\newcommand\Mtildemid{\stackrel{\, \sim}{\smash{M}\rule{0pt}{0.9ex}}}
\newcommand\Mtildelow{\stackrel{\, \sim}{\smash{M}\rule{0pt}{0.8ex}}}

\newcommand{\supp}{\mathrm{supp}}
\newcommand{\diam}{\text{diam}}

\newcommand{\Tr}{\mathrm{Tr}}
\newcommand{\HS}{\mathrm{HS}}

\newcommand{\opnorm}[1]{{\left\vert\kern-0.25ex\left\vert\kern-0.25ex\left\vert #1 
	\right\vert\kern-0.25ex\right\vert\kern-0.25ex\right\vert}}

\newcommand{\Vh}{V_h}
\newcommand{\VH}{V_H}
\newcommand{\Vf}{V_{\mathrm{f}}}
\newcommand{\VflK}{V_{\mathrm{f},\ell}^{\scalebox{0.6}{$K$}}}
\newcommand{\Vms}{V^{\mathrm{ms}}}
\newcommand{\Vmsl}{V^{\mathrm{ms}}_\ell}

\newcommand{\Th}{T_h}
\newcommand{\Nh}{N_h}
\newcommand{\NH}{\mathcal{N}_H}

\newcommand{\HdoneBochner}{L_2(\Omega; \dot{H}^1)}
\newcommand{\U}{\mathcal{U}}
\renewcommand{\H}{L_2(D)}
\newcommand{\V}{H^1_0(D)}
\newcommand{\Vdual}{H^{-1}}

\newcommand{\Rmsl}{R^{\mathrm{ms}}_\ell}
\newcommand{\Rf}{R_{\mathrm{f}}}
\newcommand{\Rfl}{R_{\mathrm{f},\ell}}
\newcommand{\RflK}{R_{\mathrm{f},\ell}^{\scalebox{0.6}{$K$}}}

\newcommand{\Pmsl}{P^{\mathrm{ms}}_\ell}

\newcommand{\Pmslh}{P_{\ell,h}^{\mathrm{ms}}}

\newcommand{\Amsl}{\Lambda^\mathrm{ms}_\ell}

\newcommand{\Tmsl}{T_\ell^\mathrm{ms}}
\newcommand{\I}{\mathfrak{I}}
\newcommand{\Ah}{\Lambda_h}

\newcommand{\Ph}{P_h}

\newcommand{\PH}{P_H}

\newcommand{\N}{N}
\newcommand{\Id}{\text{Id}}

\newcommand{\inttn}[1]{\int_{t_{n-1}}^{t_n} #1\, \mathrm{d}W(s)}
\newcommand{\inttj}[1]{\int_{t_{j-1}}^{t_j} #1\, \mathrm{d}W(s)}
\newcommand{\discd}{\bar{\partial}_t}
\newcommand{\tj}{t_j}
\newcommand{\tjprev}{t_{j-1}}
\newcommand{\tn}{t_n}
\newcommand{\tnprev}{t_{n-1}}

\newcommand{\rhon}{\rho^n}

\newcommand{\en}{e^n}

\newcommand{\ej}{e^j}

\newcommand{\rhoj}{\rho^j}

\newcommand{\Un}{U_n}
\newcommand{\Uj}{U_j}
\newcommand{\Unprev}{U_{n-1}}
\newcommand{\Xhn}{X^n_{h}}
\newcommand{\Xhnprev}{X^{n-1}_{h}}
\newcommand{\Unmsl}{U^\mathrm{ms}_{\ell,n}}

\newcommand{\Xnmsl}{X^n_{\mathrm{ms},\ell}}
\newcommand{\Xnmslprev}{X^{n-1}_{\mathrm{ms}, \ell}}
\newcommand{\Xomsl}{X^0_{\mathrm{ms}, \ell}}
\newcommand{\Xnmslm}{X^{n,m}_{\mathrm{ms}, \ell}}
\newcommand{\XnmslHJ}{X^n_{\mathrm{ms},\ell, J}}
\newcommand{\XnmslHj}{X^n_{\mathrm{ms},\ell, j}}
\newcommand{\XnmslHjprev}{X^n_{\mathrm{ms}, \ell, j-1}}
\newcommand{\XomslHj}{X^n_{\mathrm{ms}, \ell, 0}}
\newcommand{\XnmslHjm}{X^{n,m}_{\mathrm{ms},\ell, j}}
\newcommand{\XnmslHjprevm}{X^{n,m}_{\mathrm{ms}, \ell, j-1}}
\newcommand{\XomslHjm}{X^{n,m}_{\mathrm{ms}, \ell, 0}}
\newcommand{\Ekh}{E_{k,h}}
\newcommand{\Ekmsl}[1]{E^\mathrm{ms}_{k,\ell,#1}}
\newcommand{\Fmsl}[1]{F^\mathrm{ms}_{\ell,#1}}

\renewcommand{\d}{\mathrm{d}}
\newcommand{\dt}{\mathrm{d}t}
\newcommand{\ds}{\mathrm{d}s}
\newcommand{\dW}{\mathrm{d}W}
\newcommand{\dX}{\mathrm{d}X}

\newcommand{\Xh}{X_h}

\newcommand{\vms}{v_\mathrm{ms}}
\newcommand{\vf}{v_\mathrm{f}}

\newcommand{\X}{X}

\newcommand{\intt}{\int_0^t}
\newcommand{\Xo}{X_0}

\newcommand{\E}[1]{\mathbb{E}[#1]}
\newcommand{\EBig}[1]{\mathbb{E}\bigg[#1\bigg]}
\renewcommand{\P}{\mathbb{P}}
\newcommand{\R}{\mathbb{R}}

\newcommand{\HdoneBochnerNorm}[1]{\| #1 \|_{L_2(\Omega; \dot{H}^1)}}
\newcommand{\intHdoneBochnerNorm}[1]{\Big\| #1 \Big\|_{L^2(\Omega; \dot{H}^1)}}

\newcommand{\HSnorm}[1]{\| #1 \|_{\mathrm{HS}}}

\newcommand{\calB}{\mathcal{B}}
\newcommand{\calF}{\mathcal{F}}
\newcommand{\calH}{\mathcal{H}}
\newcommand{\calD}{\mathcal{D}}

\newcommand{\Kh}{\mathcal{K}_h}
\newcommand{\KH}{\mathcal{K}_H}
\newcommand{\ord}{\mu}

\definecolor{myBlue}{RGB}{30,144,255}
\definecolor{myGreen}{RGB}{69,169,0}
\definecolor{myRed}{RGB}{165,12,42} 
\definecolor{myOrange}{RGB}{225,92,22} 
\definecolor{color2}{RGB}{255, 126, 126}
\definecolor{color3}{RGB}{0, 100, 0}
\definecolor{color1}{RGB}{176, 226, 255}

\definecolor{crimson2143940}{RGB}{214,39,40}
\definecolor{darkgray176}{RGB}{176,176,176}
\definecolor{darkorange25512714}{RGB}{255,127,14}
\definecolor{forestgreen4416044}{RGB}{44,160,44}
\definecolor{lightgray204}{RGB}{204,204,204}
\definecolor{mediumpurple148103189}{RGB}{148,103,189}
\definecolor{sienna1408675}{RGB}{140,86,75}
\definecolor{steelblue31119180}{RGB}{31,119,180}
\definecolor{finefem}{RGB}{200,150,200}

\tikzset{cross/.style={cross out, draw=black, minimum size=2*(#1-\pgflinewidth), inner sep=0pt, outer sep=0pt},
cross/.default={0.6ex}}

\newtheorem{theorem}{Theorem}[section]

\newtheorem{lemma}[theorem]{Lemma}

\theoremstyle{definition}

\newtheorem{remark}[theorem]{Remark}

\pgfplotstableset{
every head row/.style={before row=\toprule, after row=\midrule},
every last row/.style={after row=\bottomrule},
}

\numberwithin{theorem}{section}
\numberwithin{equation}{section}
\numberwithin{table}{section}
\numberwithin{figure}{section}
\allowdisplaybreaks
\definecolor{darkgreen}{rgb}{0,.6,0}

\usepackage{ulem}
\begin{document}
\title[LOD for a multiscale parabolic SPDE]{Localized orthogonal decomposition for a multiscale parabolic stochastic partial differential equation}
\author[A.~Lang]{Annika~Lang$^\dagger$}
\author[P.~Ljung]{Per~Ljung$^\dagger$}
\author[A.~M{\aa}lqvist]{Axel~M{\aa}lqvist$^\dagger$}
\address{${}^{\dagger}$ Department of Mathematical Sciences, Chalmers University of Technology and University of Gothenburg, 412 96 G\"oteborg, Sweden}
\email{annika.lang@chalmers.se, perlj@chalmers.se, axel@chalmers.se}
\date{\today}
\thanks{
	Acknowledgement. The authors thank Stig Larsson for helpful comments and fruitful discussions. The work of AL was partially supported by the Swedish Research Council (VR) through grant no.\ 2020-04170, by the Wallenberg AI, Autonomous Systems and Software Program (WASP) funded by the Knut and Alice Wallenberg Foundation, and by the Chalmers AI Research Centre (CHAIR). The work of AM was partially supported by the Swedish Research Council (VR) through grant no. 2019-03517. 
}
%
\begin{abstract}

A multiscale method is proposed for a parabolic stochastic partial differential equation with additive noise and highly oscillatory diffusion. The framework is based on the localized orthogonal decomposition (LOD) method and computes a coarse-scale representation of the elliptic operator, enriched by fine-scale information on the diffusion. Optimal order strong convergence is derived. The LOD technique is combined with a (multilevel) Monte-Carlo estimator and the weak error is analyzed. Numerical examples that confirm the theoretical findings are provided, and the computational efficiency of the method is highlighted.

\end{abstract}
%
%
\maketitle
%

\section{Introduction}\label{s:intro}

We consider numerical approximations of a multiscale parabolic stochastic partial differential equation (SPDE) with additive noise. The equation takes the general form
\begin{align}
	\dX(t) + \Lambda \X(t) \, \dt = G\, \dW(t),
	\label{eq:introduction_equation}
\end{align}
with initial condition $\X(0) = \Xo$, posed on a polygonal (or polyhedral) domain $D\subset \mathbb{R}^{d}$, $d=2,3$. The diffusion operator $\Lambda$ is defined as $\Lambda := -\nabla \cdot A\nabla$, where in multiscale applications, the coefficient~$A$ varies rapidly in space. Such effects arise when modeling the physical behavior of, for instance, porous media or composite materials.

Computing samples and quantities of interest such as moments of the solution to~\eqref{eq:introduction_equation} becomes computationally heavy since convergence is only achieved on very fine space grids resolving the rapidly varying diffusion coefficient. To overcome this bottleneck and to guarantee convergence on coarse grids and of more efficient multilevel Monte-Carlo methods, we use the localized orthogonal decomposition (LOD) method for the first time in the context of SPDEs. In order to show convergence in the stochastic setting, we derive new error estimates for LOD methods in the deterministic setting that take the roughness of stochastic equations into account.

In general simulation of stochastic partial differential equations is of importance to a wide range of real-life applications. This includes models in physics, chemistry, biology and mathematical finance (see, e.g.,~\cite{LorPS14} for more applications, and~\cite{DapZ14, liu2015stochastic, LotR17} for several, more precise examples). In particular, the equation in~\eqref{eq:introduction_equation} with multiscale effects arises, for instance, when modeling heat flow in an inhomogeneous (e.g., composite) material with uncertainties in measurements of the source data. The theory of these equations is by now well-established. For results on existence and uniqueness, regularity, asymptotic behavior, and further properties of the solution, we refer to e.g.~\cite{DapZ14, Kru14}.

From a numerical standpoint, SPDEs can be approximated by means of the standard finite element discretization. In fact, in~\cite{Yubin04, Yub05}, the author derives convergence for the strong error of~\eqref{eq:introduction_equation} (see~\cite{Kru14} for extension to the semi-linear case), i.e., the error in root-mean square or more general in~$L_p$ over the probability space. However, the results in these papers rely heavily on spatial $H^2$-regularity of the solution, which is ill-suited for multiscale applications. More precisely, if we denote by $\varepsilon$ the scale at which the diffusion $A$ oscillates, then the error depends on $\|X(t)\|_{H^2}\sim \varepsilon^{-1}$ for the deterministic part of~\eqref{eq:introduction_equation}. In other words, this means that the corresponding finite element mesh width $h$ must satisfy $h < \varepsilon$, i.e., resolve the variations of $A$, before reaching the region of convergence. In terms of computation, this quickly becomes challenging, both in complexity and memory consumption.

In recent years, research has begun to emerge in the area of multiscale SPDEs. For instance, in~\cite{HaiK12}, the authors study analytical homogenization of a parabolic SPDE, and~\cite{BloHP07} reduces the complexity of computing solutions to~\eqref{eq:introduction_equation} by means of so-called amplitude equations. From a numerical perspective, \cite{Bre12, Bre13, Bre20} consider homogenization and averaging principles for stochastic multiscale reaction-diffusion equations and apply a numerical scheme based on the framework of heterogeneous multiscale methods (see, e.g.,~\cite{AbdEEV12, EE03}). However, since the analysis is based on analytical homogenization, it requires additional assumptions on the diffusion coefficient, such as scale separation and periodicity. To overcome these restrictions, several methods based on numerical homogenization have been developed for deterministic PDEs, where prominent examples include the LOD method~\cite{HenP13, MalP14}, generalized multiscale finite element method (GMsFEM)~\cite{BabL11, BabO83, EfeGH13}, gamblets~\cite{Owh17}, and the Super-LOD method~\cite{HauP21, FreHP21}. For a general overview on numerical homogenization, see also~\cite{AltHP21, MalP20, OwhS19}. We use the LOD method for the first time in the context of SPDEs.

The LOD method was originally presented in~\cite{MalP14} and is based on the variational multiscale method (see, e.g.,~\cite{Hug95, HugS96, HugFMQ98}). By now, the technique is well-established and has been thoroughly analyzed for a wide range of problems. Of particular interest for the work in this paper is the LOD method applied to parabolic-type equations, mainly covered in~\cite{MalP18, LjuMM22, MalP17, AltCMPP18}. For a general and rigorous introduction to the LOD method, we refer to~\cite{MalP20}. In short, the LOD technique utilizes a subspace decomposition, where the solution space is split into a coarse-scale and a fine-scale part, respectively. On the fine scale, comprehensive information about the diffusion is computed, and subsequently incorporated into the coarse scale to construct a modified version of the standard finite element space. Consequently, this modified space yields enhanced approximation properties while the computational cost maintains comparable to that of standard finite elements on coarse grids.

In terms of computational complexity, we remark that constructing the modified space is feasible, but still challenging due to its fine-scale dependency. However, we emphasize that the LOD method is significantly improved when applied to evolution equations. This follows naturally since the fine-scale features, once computed, can be re-used for each time step in the temporal partition. Moreover, a well-known characteristic of the LOD methodology is the ability to compute solutions for multiple different source functions, since the definition of the modified space solely depends on the diffusion operator. In fact, the re-usability gains an additional level of interest in the area of SPDEs. Here, solutions are characterized by their different moments, generally approximated by many samples using (multilevel) Monte-Carlo estimators (see~\cite{BarLS13, LanP17}). A well-known issue with this approach is the slow convergence of Monte-Carlo methods in the number of samples, requiring us to solve~\eqref{eq:introduction_equation} for a large number of realizations. In total, by computing fine-scale features once, our (coarse-scale) modified space can be re-used in each time step for any such realization of~\eqref{eq:introduction_equation}. Furthermore, the LOD technique can be applied to several discretization layers, such that it easily can be combined with the multilevel Monte-Carlo framework. Later, in Section~\ref{s:numerical_examples}, we provide additional emphasis and details on the computational complexity for the main method of this paper.

In combination with the spatial discretization, we further apply a standard backward Euler time scheme to define a fully discretized numerical method. For the full scheme, we derive a priori strong convergence of optimal order with respect to the coarse mesh size. In general, the analysis is based on techniques from finite element theory on parabolic PDEs and SPDEs (see \cite{thomee97} and \cite{Yubin04, Yub05, Kru14}, respectively) and the LOD method for parabolic PDEs (see \cite{MalP18}). The theoretical results are confirmed by numerical examples.

The remaining paper is outlined as follows. In Section~\ref{s:model_problem}, we introduce the model problem and discuss results on regularity and standard finite elements. Section~\ref{s:lod} defines the LOD methodology and composes the main numerical scheme of this paper. In Section~\ref{s:error_analysis}, the strong error is analyzed, for which optimal order convergence in the spatial sense is derived. This is followed by a discussion on the weak error and its estimation by Monte-Carlo methods in Section~\ref{s:weak_error}. In Section~\ref{s:numerical_examples}, we conclude with numerical examples that confirm the theoretical findings for the strong and weak error, respectively, and discuss the computational benefits of the LOD method applied to SPDEs.

\section{Model problem}\label{s:model_problem}

We consider the stochastic partial differential equation
\begin{alignat}{2}
	\dX(t) + \Lambda \X(t)\, \dt &= G\, \dW(t), \quad &&\text{in $D\times (0,T]$}, \label{eq:spde} \\
	\X(t) &= 0, \quad &&\text{on $\partial D \times (0,T]$}, \nonumber \\
	\X(0) &= \Xo, \quad &&\text{in $D$}, \nonumber
\end{alignat}
where $T > 0$ and $D$ is a polygonal (or polyhedral) domain in $\R^d$, $d=2,3$, with boundary $\partial D$ on which we assume homogeneous Dirichlet boundary conditions, and $\Xo$ is the (possibly stochastic) initial value. In this paper, we assume an operator $\Lambda$ of the form $\Lambda := -\nabla \cdot A \nabla$, where the diffusion coefficient $A$ is highly oscillatory in space, but independent of time. Moreover, the diffusion $A\in L_\infty(D; \R^{d\times d})$ is assumed to be symmetric and uniformly elliptic, i.e.,
\begin{align}
0 < \alpha_- := \essinf_{x\in D} \inf_{v\in \R^d\backslash\{0\}} \frac{A(x)v\cdot v}{v\cdot v} < \esssup_{x\in D} \sup_{v\in \R^d\backslash\{0\}} \frac{A(x)v\cdot v}{v\cdot v} =: \alpha_+ < +\infty.
\end{align}

In this section, we mainly focus on results for the solution to~\eqref{eq:spde}, as well as discuss known results for an approximate solution by means of the finite element method. First, we present preliminaries and notation that are used throughout the paper.

\subsection{Preliminaries and notation}

 Let $L(\U;\calH)$ denote the space of linear bounded operators between two separable Hilbert spaces $\U$ and $\calH$, with the short-hand notation $L(\calH) := L(\calH;\calH)$ in the case $\U=\calH$. Furthermore, denote by $L_N^+(\U)$ the space of all nonnegative, symmetric, nuclear operators on $\U$. We assume $W$ to be a $Q$-Wiener process with covariance operator $Q\in L_N^+(\U)$, defined on a filtered probability space $(\Omega, \calF, \{\calF_t\}_{t>0}, \mathbb{P})$ satisfying the usual conditions. It is well-known (see, e.g.,~\cite[Proposition 4.3]{DapZ14}) that such a process can be expressed by its Karhunen--Loève expansion
 \begin{align}
 	W(t) = \sum_{i=1}^\infty \sqrt{\lambda_i} \beta_i(t)e_i,
 	\label{eq:Karhunen-Loeve}
 \end{align}
where $\{e_i\}_{i=1}^\infty$ is an eigenbasis of $Q$ with corresponding eigenvalues $\{\lambda_i\}_{i=1}^\infty$, and $\{\beta_i\}_{i=1}^\infty$ is a sequence of mutually independent, real-valued Brownian motions. Moreover, since $Q\in L_N^+(\U)$, it holds that $Q$ is a trace class operator, such that
\begin{align}
	\Tr(Q) := \sum_{i=1}^\infty \langle Qe_i, e_i \rangle_\U < +\infty.
\end{align}
 
Next, we introduce $\HS(\U;\H)$ as the space of Hilbert--Schmidt operators between the two Hilbert spaces $\U$ and $\H$, with norm
\begin{align}
	\HSnorm{T}^2 := \sum_{i=1}^\infty \|T\phi_i\|^2,
\end{align}
where $\{\phi_i\}_{i=1}^\infty$ is an arbitrary ON-basis of $\U$, and $\|\cdot\|$ denotes the standard $\H$-norm. Furthermore, we use the space $L_2^0 := \HS(Q^{1/2}(\U);\H)$, with norm
\begin{align}
	\|\psi\|^2_{L_2^0} := \sum_{i=1}^\infty \|\psi Q^{1/2} \phi_i\|^2.
\end{align}

Next, since the operator $\Lambda$ (with incorporated homogeneous Dirichlet boundary conditions) is self-adjoint, the spectral theorem ensures a sequence of positive, non-decreasing eigenvalues $\{\lambda_i\}_{i=1}^\infty$ with corresponding eigenfunctions $\{\varphi_i\}_{i=1}^\infty$ that form an orthonormal basis for $\H$. For $s\geq 0$, we define fractional powers of $\Lambda$ by
\begin{align}
	\Lambda^{s/2}v := \sum_{i=1}^\infty \lambda_i^{s/2} (v, \varphi_i)\varphi_i,
\end{align}
and introduce the space $\dot{H}^s := \calD(\Lambda^{s/2})$, with norm given by
\begin{align}
	|v|_s := \|\Lambda^{s/2}v\| = \Big(\sum_{i=1}^\infty \lambda_i^s (v,\varphi_i)^2 \Big)^{1/2},\label{eq:s_norm_definition}
\end{align}
where $\calD(\Lambda^{s/2})$ denotes the domain of the operator $\Lambda^{s/2}$. In the frequently used case $s=1$, it holds that $\dot{H}^1 = H^1_0$ (see~\cite[Lemma~3.1]{thomee97}), and $|v|_1 = a(v,v)^{1/2}$, where $a(\cdot, \cdot)$ is the bilinear form obtained from $\Lambda$, defined by the relation
\begin{align}
	\langle \Lambda v, w \rangle = a(v,w), 
\end{align}
for all $v,w\in \V$. Here, $\langle \cdot, \cdot \rangle:= {}_{\Vdual}\langle \cdot, \cdot \rangle_{H^1_0}$ denotes the dual pairing, characterized by the Gelfand triple $\V \subset \H \cong \H \subset H^{-1}(D)$. Moreover, the operator $-\Lambda$ is the infinitesimal generator of an analytic semigroup (see~\cite[Appendix~B.2]{Kru14}). At last, we define the Bochner space
\begin{align}
	L_2(\Omega; \dot{H}^s) := \Big\{ v \in \dot{H}^s \colon \E{|v|^2_s} = \int_\Omega |v(\omega)|^2_s\, \mathrm{d}\P(\omega) < +\infty  \Big\},
\end{align}
with norm $\|v\|_{L_2(\Omega; \dot{H}^s)} := (\E{|v|^2_s})^{1/2}$. Throughout the paper, we frequently abbreviate the Bochner space by neglecting the spatial domain and write, i.e., $L_2(\Omega; L_2) := L_2(\Omega; L_2(D))$.

\subsection{Regularity}

We continue by analyzing the regularity of the solution to our model equation~\eqref{eq:spde}. For this purpose, let $E$ denote the semigroup generated by the operator $-\Lambda$, i.e,
\begin{align}
	E(t)v = e^{-t\Lambda}v = \sum_{i=1}^{\infty} e^{-\lambda_i t}\hat{v}_i\varphi_i, \label{eq:heat_semi_group}
\end{align}
where $\hat{v}_i := (v,\varphi_i)$, and $\{\lambda_i, \varphi_i\}_{i=1}^\infty$ are the eigenpairs of the operator $\Lambda$ (see~\cite{KruL11}).

\begin{remark}\label{rem:elliptic_regularity}
	We emphasize that although the semigroup generated by $-\Lambda$ is analytic, there are several properties of the semigroup that need to be handled with care. This is since many of the properties implicitly utilize elliptic regularity, where the constant is dependent on the high variations in $A$, which in the case of multiscale materials may become prohibitively large.
\end{remark}

The necessary properties for the semigroup are stated in the following lemma.
\begin{lemma}\label{lem:semigroup_properties}
	The semigroup $E$ defined in \eqref{eq:heat_semi_group} is a $C_0$-semigroup of contractions, i.e., it holds that $\|E(t)\|_{L(L_2(D))} \leq 1$ for $t\geq 0$.
	Moreover, $E$ satisfies the following:
	\begin{enumerate}
		\item For every $v\in \mathcal{D}(\Lambda^\alpha)$, $\alpha \in \mathbb{R}$, the operators $\Lambda^\alpha$ and $E(t)$ commute, i.e., for all $t\geq 0$,
		\begin{align}
			E(t)\Lambda^\alpha v = \Lambda^\alpha E(t)v.	
		\end{align}
		\item For $v\in L_2(D)$, it holds that
		\begin{align}
			\int_0^t \| \Lambda^{1/2} E(s) v \|^2\, \ds \leq \frac{1}{2}\|v\|^2, \quad t\geq 0.
		\end{align}
		\item For $v\in \dot{H}^s$, it holds that
		\begin{align}
			|D^\ell_t E(t)v|_s \leq Ct^{-\ell}|v|_s, \quad s = 0,1,
			\label{eq:semigroup_derivative}
		\end{align}
		for $t>0$, where $D^\ell_t$ denotes the $\ell$'th order time derivative, and the constant $C$ is independent of the variations in the diffusion $A$.
	\end{enumerate}
\end{lemma}
\begin{proof}
	First of all, using the expansion \eqref{eq:heat_semi_group} we immediately have 
	\begin{align}
		\|E(t)\|^2_{L(L_2(D))} = \sup_{v\in L_2(D)\backslash\{0\}} \frac{\|E(t)v\|^2}{\|v\|^2} = \frac{\sum_{i=1}^\infty e^{-2\lambda_it}\hat{v}_i^2}{\sum_{i=1}^\infty \hat{v}_i^2} \leq 1,
	\end{align}
	since the eigenvalues of $\Lambda$ are positive. 
	
	The fact that $E(t)$ and $\Lambda^\alpha$ commute follows from an argument based on spectral theory for the operator $\Lambda^\alpha$ and is proven in \cite[Theorem 6.13]{pazy1983semigroups}. 
	
	For the remaining properties, we first note that
	\begin{align}
	|D^\ell_t E(t)v|_1^2 &= \|\Lambda^{1/2}D^\ell_t E(t)v\|^2 = \sum_{i=1}^\infty \lambda_i (D^\ell_t E(t)v, \varphi_i)^2 = \sum_{i=1}^\infty \lambda_i^{2\ell+1} e^{-2\lambda_it}\hat{v}_i^2,
	\end{align}
	where the second equality is the expression for the $\dot{H}^1$-norm defined in~\eqref{eq:s_norm_definition}, and the last follows from the eigenvalue expansion~\eqref{eq:heat_semi_group}. The property \textit{(2)} then follows by inserting this into the integral, such that
	\begin{align}
		\int_0^t \|\Lambda^{1/2}E(s)v\|^2\, \ds = \sum_{i=1}^\infty \hat{v}^2_i \int_0^t \lambda_ie^{-2\lambda_is}\, \ds = \sum_{i=1}^\infty \frac{\hat{v}^2_i}{2}  (1-e^{-2\lambda_it}) \leq \frac{1}{2}\|v\|^2.
	\end{align}
	For the property \textit{(3)}, the result in $|\cdot|_1$-norm follows by
	\begin{align}
		|D^\ell_t E(t)v|^2_1 &= \sum_{i=1}^\infty \lambda_i^{2\ell}e^{-2\lambda_it} \lambda_i \hat{v}_i^2 \leq Ct^{-2\ell}\sum_{i=1}^\infty \lambda_i\hat{v}_i^2 = Ct^{-2\ell}|v|^2_1,
	\end{align}
	where the inequality follows since $s^{k}e^{-s} \leq C$. The result in standard $L_2(D)$-norm follows similarly.
\end{proof}

Using the semigroup, we know from, e.g.,~\cite{DapZ14}, that the unique mild solution of \eqref{eq:spde} is given by
\begin{align}
	\X(t) = E(t)\Xo + \intt E(t-s) G\, \dW(s).
	\label{eq:mild_solution}
\end{align}
We may now derive the following result for the regularity of the mild solution. The proof is similar to the regularity result proven in~\cite[Theorem 2.1]{Yubin04}, but adapted to our setting and keeping track of the constants such that no blow-up appears due to the multiscale effects.

\begin{lemma}
	Assume that $G \in L_2^0$ and $\Xo \in \HdoneBochner$. Then, the mild solution $\X$ in~\eqref{eq:mild_solution} satisfies
	\begin{align}
		\HdoneBochnerNorm{\X(t)} \leq \HdoneBochnerNorm{\Xo} + \frac{1}{\sqrt{2}}\|G\|_{L_2^0}.
	\end{align}
\end{lemma}
\begin{proof}
	At first, observe that
	\begin{align}
		\HdoneBochnerNorm{\X(t)}^2 = \HdoneBochnerNorm{E(t)\Xo}^2 + \intHdoneBochnerNorm{\intt E(t-s)G\, \dW(s)}^2,
	\end{align}
	since the mixed term vanishes due to the zero mean property of the Itô integral. For the first term, we note that
	\begin{align}
		\HdoneBochnerNorm{E(t)\Xo}^2 = \E{\|E(t)\Lambda^{1/2}\Xo\|^2} \leq \E{\|\Lambda^{1/2}\Xo\|^2} = \HdoneBochnerNorm{\Xo}^2.
	\end{align}
	Here, we first used the fact that $E(t)$ and $\Lambda^{1/2}$ commute, followed by the fact that $E$ is a semigroup of contractions, both stated in Lemma~\ref{lem:semigroup_properties}. For the second term, we get the bound
	\begin{align}
		\EBig{\Big\| \int_0^t E(t-s)G\, \dW(s) \Big\|^2_{\dot{H}^1}} &= \EBig{\Big\| \int_0^t \Lambda^{1/2}E(t-s)G\, \dW(s) \Big\|^2} \\
		&= \EBig{\int_0^t \| \Lambda^{1/2}E(t-s)G \|^2_{L_2^0}\, \ds} \\
		&\leq \frac{1}{2} \| G \|^2_{L_2^0},
	\end{align}
	where we have used the Itô isometry (see~\cite[Section~2.2]{Kru14}), and the semigroup property~\textit{(2)} from Lemma~\ref{lem:semigroup_properties}.
\end{proof}

\subsection{Finite element method}

We close this section by introducing the fully discrete Galerkin finite element approximation of~\eqref{eq:spde} and known results from the literature for the error between the exact solution and its finite element approximation. Let $\{\Kh\}_{h > 0}$ be a family of shape regular elements that form a partition of the spatial domain $D$. For any element $K\in \Kh$, we denote the mesh size of the element by $h_K := \diam(K)$, and the largest diameter in the partition by $h := \max_{K\in \Kh} h_K$. We define the standard finite element space consisting of continuous piecewise linear polynomials as
\begin{align}
S_h := \big\{v\in C(\bar{D}) : \ v\big|_K \text{ is a polynomial of partial degree $\leq 1$}, \forall K\in \Kh \big\},
\end{align}
and let $\Vh := S_h \cap H^1_0$, with dimension denoted by $N_h$. The semi-discrete version of \eqref{eq:spde} states: find $\Xh(t) \in \Vh$ for $t\in (0,T]$, such that
\begin{align}
\d \Xh(t)  + \Ah \Xh(t)\, \dt = \Ph G\, \dW(t)
\label{eq:spde_h}
\end{align}
with initial value $\Xh(0) = \Ph \Xo$, and where $\Ah:\Vh \rightarrow \Vh$ is the discrete version of the operator $\Lambda$, defined by the relation
\begin{align}
(\Ah v, w) = a(v, w),
\label{eq:Ah}
\end{align}
for all $v,w\in \Vh$, and $\Ph:L_2(D) \rightarrow \Vh$ denotes the standard $L_2(D)$-projection onto $\Vh$, i.e, for all $w\in \Vh$, it holds that
\begin{align}
	(\Ph v, w) = (v, w).
\end{align}

For the discretization of the temporal domain $[0,T]$, let $0=:t_0 < t_1 < \ldots < t_N := T$ be a partition with uniform time step $k = \tn - \tnprev$. We remark that a uniform partition is chosen for simplicity, and refer to classical finite element results for the choice of varying time step. Let $\Xhn$ be the approximation of $\X(\tn)$. The backward Euler scheme for~\eqref{eq:spde_h} is defined as
\begin{align}
	\Xhn - \Xhnprev + k\Ah \Xhn = \inttn{\Ph G}.
	\label{eq:original_BE}
\end{align}
Alternatively, let $\Ekh:= (I+k\Ah)^{-1}$, and we can write \eqref{eq:original_BE} as
\begin{align}
	\Xhn = \Ekh \Xhnprev + \inttn{\Ekh \Ph G}.
	\label{eq:fully_discrete_galerkin}
\end{align}
Furthermore we iterate this expression which yields a discrete analogy to Duhamel's principle, namely
\begin{align}
	\Xhn = \Ekh^n \Ph \Xo + \sum_{j=1}^n \inttj{\Ekh^{n-j+1}\Ph G}.
\end{align}
For the error analysis in Section~\ref{s:error_analysis}, we require the following lemma for $\Ekh$. Here, $\discd$ denotes the discrete time derivative, such that $\discd f^n = (f^n - f^{n-1})/k$ for any temporally discrete function $f^n$.

\begin{lemma}\label{lem:Ekh_estimate}
	Let $\Ekh= (I+k\Ah)^{-1}$, with $\Ah$ as defined in~\eqref{eq:Ah}. Then, for any function $v\in \dot{H}^s$, it holds for $n\geq \max\{1,\ell\}$
	\begin{align}
	|\discd^\ell \Ekh^nv|_q \leq C\tn^{(s-q)/2-\ell} |v|_s, \quad s,q \in \{0,1\}, \ \ell = 1,2.
	\label{eq:discrete_semigroup_bound}
	\end{align}
	For $\ell=0$, the inequality~\eqref{eq:discrete_semigroup_bound} holds for $0\leq s \leq q \leq 1$.
\end{lemma}
\begin{proof}
	First of all, it holds that $\Ah$ is self-adjoint, positive definite and defined on the finite-dimensional space $\Vh$. Therefore, there exists a finite set of positive eigenvalues $\{\lambda_j^h\}_{j=1}^{\Nh}$ with corresponding eigenvectors $\{\varphi_j^h\}_{j=1}^{\Nh}$ satisfying
	\begin{align}
	a(\varphi_j^h, v) = \lambda_j^h(\varphi_j^h, v),
	\end{align}
	for all $v\in \Vh$, such that $\Vh = \text{span}(\{\varphi_j^h\}_{j=1}^{\Nh})$ (see, e.g.,~\cite[Chapter~6]{LarT03}). Consequently, we can write
	\begin{align}
	\Ekh^nv = \sum_{j=1}^{\Nh} r(k\lambda_j)^n \hat{v}^h_j\varphi_j^h,
	\label{eq:discrete_eigenfunction_expansion}
	\end{align}
	where we have denoted $\hat{v}^h_j = (v, \varphi_j^h)$ and where $r(z) = 1/(1+z)$. Let $\sigma(\Lambda_h)$ denote the spectrum of the operator $\Lambda_h$. Note that, for $\ell= 1,2$ and $\lambda\in\sigma(\Lambda_h)$, we have
	\begin{align}
		\begin{split}
			\lambda^{2\ell-1}r(k\lambda)^{2n} &\leq \frac{\lambda^{2\ell-1}}{(1+k\lambda)^n} \leq \frac{\lambda^{2\ell-1}}{1 + \binom{n}{1}k\lambda + \cdots + \binom{n}{2\ell-1}(k\lambda)^{2\ell-1} + \cdots (k\lambda)^n} \\
			&\leq \frac{\lambda^{2\ell-1}}{\binom{n}{2\ell-1}(k\lambda)^{2\ell-1}} \leq \frac{1}{\frac{n^{2\ell-1}}{(2\ell-1)^{2\ell-1}}k^{2\ell-1}}  = \frac{(2\ell-1)^{2\ell-1}}{\tn^{2\ell-1}} = C\tn^{1-2\ell},
		\end{split}\label{eq:lambda_bound}
	\end{align}
	where the constant behaves nicely for small $\ell$. We show~\eqref{eq:discrete_semigroup_bound} in the case $q=0$, $s=1$, $\ell=1,2$. The remaining cases follow similarly. First, we note that that
	\begin{align}
	\discd r(k\lambda_i)^n
	&= \frac{1}{k}\big( r(k\lambda_i)^n - r(k\lambda_i)^{n-1} \big)= \frac{1}{k}\bigg( \frac{1}{(1 + k\lambda_i)^n} - \frac{1}{(1+k\lambda_i)^{n-1}} \bigg) =  -\lambda_i r(k\lambda_i)^n.
	\end{align}
	Using this, in combination with the expansion~\eqref{eq:discrete_eigenfunction_expansion}, we obtain
	\begin{align}
	\|\discd^\ell \Ekh^n v\|^2 
	&= (\discd^\ell \Ekh^n v, \discd^\ell \Ekh^n v) \\
	&= -a(\discd^{\ell-1} \Ekh^n v, \discd^\ell \Ekh^n v) \\
	&= -a\Big(\discd^{\ell-1}\sum_{i=1}^{\Nh} r(k\lambda_i^h)^n \hat{v}^h_i\varphi_i^h, \discd^\ell \sum_{j=1}^{\Nh} r(k\lambda_j^h)^n \hat{v}^h_j\varphi_j^h \Big) \\
	&= \sum_{i=1}^{\Nh} \sum_{j=1}^{\Nh} (\lambda_i^h)^{\ell-1} (\lambda_j^h)^\ell r(k\lambda_i^h)^nr(k\lambda_j^h)^n \hat{v}^h_i \hat{v}^h_j a(\varphi_i^h, \varphi_j^h) \\
	&= \sum_{i=1}^{\Nh} (\lambda_i^h)^{2\ell-1}r(k\lambda_i^h)^{2n}(\hat{v}^h_i)^2\lambda_i^h \\
	&\leq C\tn^{-(2\ell-1)}\sum_{i=1}^{\Nh} (\hat{v}^h_i)^2 \lambda_i^h = C\tn^{1-2\ell}|v|^2_1,
	\end{align}
	where the last inequality follows from~\eqref{eq:lambda_bound}.
\end{proof}

If sufficient regularity is assumed on the initial value $\Xo$ and the covariance operator $Q$, it holds that the finite element approximation $\Xhn$ from~\eqref{eq:fully_discrete_galerkin} converges to the exact solution of~\eqref{eq:spde} quadratically in space and linearly in time. In fact, by following the proof of~\cite[Theorem~3.14]{Kru14} and adjusting the calculations to our model problem, we have the following theorem. 

\begin{theorem}\label{thm:strong_convergence_h}
	Let $\Xhn$ and $X(\tn)$ be solutions to~\eqref{eq:fully_discrete_galerkin} and~\eqref{eq:spde}, respectively. Assume that $\|\Lambda^{(\ord-1)/2}G\|_{L_2^0} < +\infty$ and $\Xo\in L_2(\Omega; \dot{H}^\ord)$, for some $\ord\in [1,2]$. Then, the error satisfies 
	\begin{align}
	\|\Xhn - X(\tn)\|_{L_2(\Omega;L_2)} \leq C(k^{\ord/2} + h^\ord) \big( \|\Xo\|_{L_2(\Omega; \dot{H}^\ord)} + \|\Lambda^{(\ord-1)/2}G\|_{L_2^0} \big).\label{eq:fe_strong_error}
	\end{align}
\end{theorem}

Note that in~\cite[Theorem~3.14]{Kru14} the temporal convergence order is limited to $1/2$, which is a consequence of considering multiplicative noise. In the case of additive noise, the Euler--Maruyama scheme approximates the stochastic integral exactly, hence removing this limit.

The theorem states optimal order convergence for the finite element approximation of~\eqref{eq:spde}. However, the estimate heavily relies on elliptic regularity, and as stated in Remark~\ref{rem:elliptic_regularity}, this consequently leads to the constant $C$ being prohibitively high. To emphasize this, we assume the operator $\Lambda$ to oscillate at a scale of $\varepsilon$. Then, since the constant depends on the second order spatial derivative of $X$, it holds that $C \sim \varepsilon^{-1}$ (see~\cite[Chapter~2]{MalP20} for a detailed discussion). In practice, this means that the optimal order convergence in~\eqref{eq:fe_strong_error} is only valid once the mesh size is sufficiently small, i.e., when $h^\ord < \varepsilon$, which quickly becomes computationally challenging with decreasing $\varepsilon$.

The purpose of this paper is to present an alternative approach to approximating~\eqref{eq:spde} by using the so-called localized orthogonal decomposition (LOD) method. In the subsequent section, we demonstrate the approach of the method, and thereafter continue by deriving a similar optimal order estimate as~\eqref{eq:fe_strong_error}, where the constant is independent of any variations present in the operator $\Lambda$.

\section{Localized orthogonal decomposition}
\label{s:lod}

This section is dedicated to the development of a multiscale method specifically designed to approximate \eqref{eq:spde}. First of all, we introduce necessary notation for the discretization. In similarity to $\Vh$, we define $\VH$ but for coarser mesh sizes $H > h$. The mesh size $h$ is assumed to be fine enough to resolve the variations in the diffusion $A$, while $H$ is the mesh size of an under-resolved coarse scale. Moreover, the corresponding family of partitions $\{\KH\}_{H>h}$ is assumed to be both shape-regular and quasi-uniform. The set of interior nodes in $\VH$ is denoted by $\NH$, and by $\phi_i$ we denote the standard P1-FEM basis function for a node $x_i \in \NH$, such that $\VH = \text{span}(\{\phi_i\}_{i\in \NH})$. At last, we assume $\Kh$ to be a refinement of $\KH$, such that we have $\VH \subset \Vh$. In general, the assumption on nested meshes is not necessary, but done for simplicity. For more information on LOD with non-nested grids and its applications, we refer to~\cite{MalP15, EdeGHKM22}.

\subsection{Ideal method}

The goal of this section is to construct a generalized finite element space $\Vms$, which is of the same dimension as $\VH$, but with improved approximation properties. The construction is done by incorporating fine-scale information about the diffusion coefficient into $\VH$. For this purpose, we require an interpolant $\I : \Vh \rightarrow \VH$ that has the projection property $\I = \I \circ \I$, and moreover satisfies the interpolation estimate
\begin{align}
H^{-1} \|(\Id - \I)v\|_{L_2(K)} + \|\nabla \I v\|_{L_2(K)} \leq \hat{C}_\I \|\nabla v\|_{L_2(\N(K))},
\label{eq:local_interpolant_estimate}
\end{align}
for any $v\in\Vh$, where $\N(K)$ is the neighborhood of $K\in \KH$, i.e.,
\begin{equation}
	N(K) := \{\bar{S} \in \KH: \bar{S} \cap \bar{K} \neq \emptyset\}.
	\label{eq:patch}
\end{equation}
An illustration of $N(K)$ can be seen in Figure~\ref{fig:pathces}. For a shape-regular mesh, the estimate~\eqref{eq:local_interpolant_estimate} can be further summed into the global bound
\begin{align}
H^{-1}\|(\Id - \I)v\| + \|\nabla \I v\| \leq C_\I \|\nabla v\|, 
\end{align}
for any $v\in\Vh$, where $C_\I$ depends on the interpolation constant $\hat{C}_\I$ and the shape-regularity of the mesh.

For the method proposed in this paper, it is not essential to make an explicit characterization of $\I$. However, a common choice, which moreover will be used in the numerical examples, is to take $\I= \pi_H \circ \Pi_H$, where $\Pi_H$ is the piecewise $L_2(D)$-projection onto the space of piecewise affine functions $P_1(\KH)$, and $\pi_H$ is an averaging operator such that for $v\in P_1(\KH)$, we have
\begin{align}
	(\pi_Hv)(x) := \sum_{K\in\KH: x\in \bar{K}} (v\mathds{1}_K)(x) \cdot \frac{1}{\text{card}\{K' \in \KH: x\in \bar{K'}\}},
\end{align}
for any vertex $x$ of $\KH$. For more information on choices of the interpolant, we refer to~\cite{EngHMP19}, and for a proof of~\eqref{eq:local_interpolant_estimate} for our particular choice, see, e.g., \cite{Bre94, ErnG17, Osw93}.

Next, let the kernel of $\I$ define the fine-scale space $\Vf := \ker(\I)$, such that $\Vf$ recovers the features which $\I$ is unable to project onto the coarse-scale space $\VH$. Consequently, this leads to a decomposition of the solution space, namely
\begin{align}
\Vh = \VH \oplus \Vf.
\end{align}
The goal of the LOD method is to create a similar decomposition, where the coarse-scale space contains information of the diffusion on the fine scale. For this purpose, we define the Ritz-projection $\Rf: \Vh \rightarrow \Vf$, such that $\Rf v \in \Vf$ solves
\begin{align}
a(\Rf v, w) = a(v, w),
\label{eq:global_ritz_projection}
\end{align}
for all $w\in\Vf$. We may now define our multiscale space as $\Vms:= \VH - \Rf \VH$, with basis $\{\phi_i - \Rf \phi_i\}_{i\in \NH}$. Here, the basis correction $\Rf \phi_i$ contains information about the diffusion coefficient on the fine scale, and therefore contributes with such information in the space $\Vms$, while we still maintain the property that $\dim(\Vms) = \dim(\VH)$. Consequently, we manage to obtain better approximation properties, while the corresponding matrix system remains equally large. For an illustration of the corrected basis $\Rf\phi_i$, and the modified basis function $\phi_i - \Rf\phi_i$, see Figure~\ref{fig:newbasis}. Moreover, note that we have the decomposition 
\begin{align}
\Vh  = \Vms \oplus \Vf,
\end{align}
i.e., each function $v_h \in \Vh$ can be decomposed as $v_h = \vms + \vf$, where $\vms \in \Vms$ and $\vf\in \Vf$ are orthogonal with respect to the inner product $a(\cdot, \cdot)$.

\begin{figure}
	\centering
	\begin{subfigure}[b]{0.4\textwidth}
		\includegraphics[width=\textwidth]{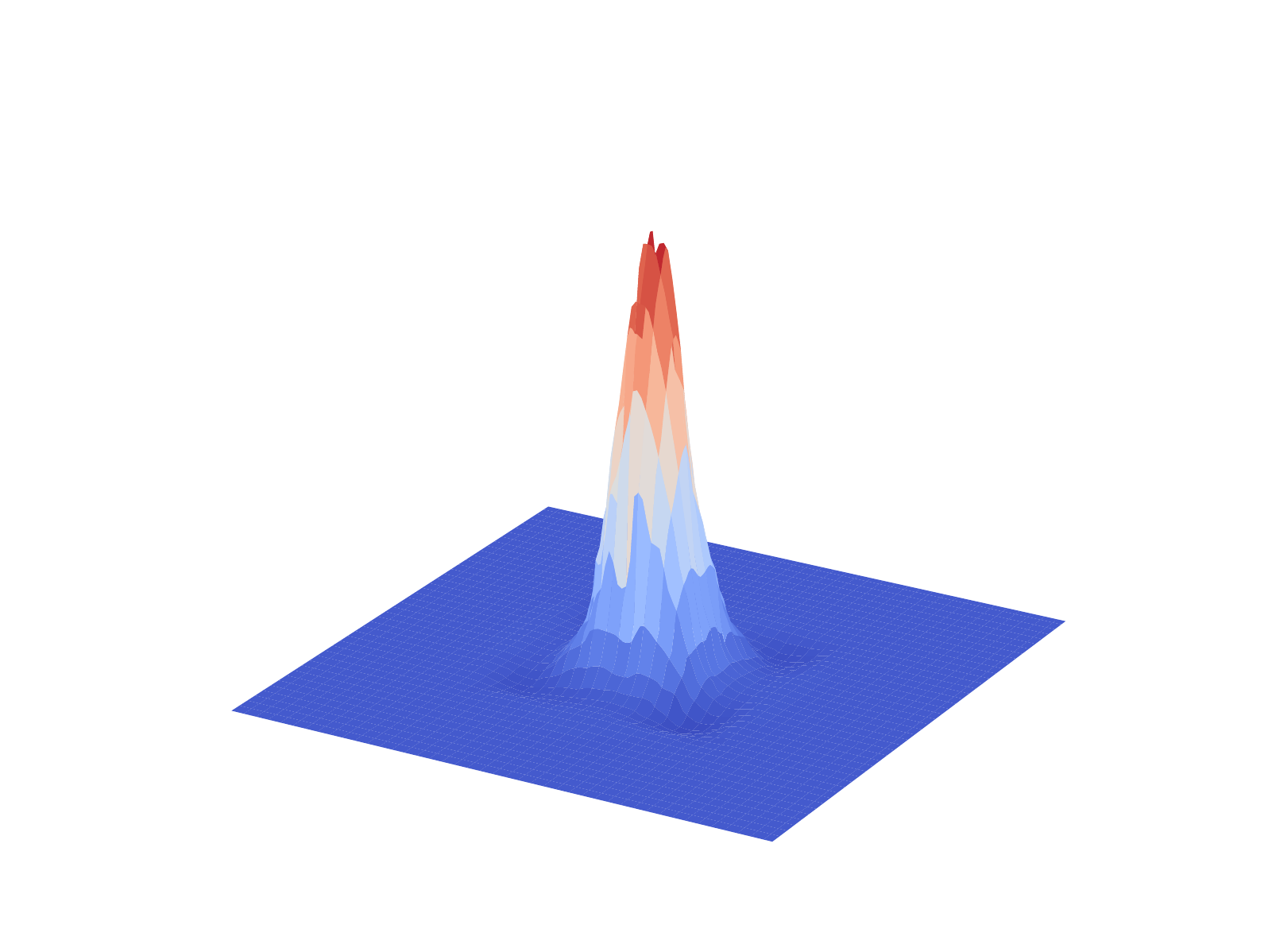}
		\caption{$\phi_i - \Rf\phi_i$.}
		\label{fig:msbasis}
	\end{subfigure}
	~ 
	\begin{subfigure}[b]{0.4\textwidth}
		\includegraphics[width=\textwidth]{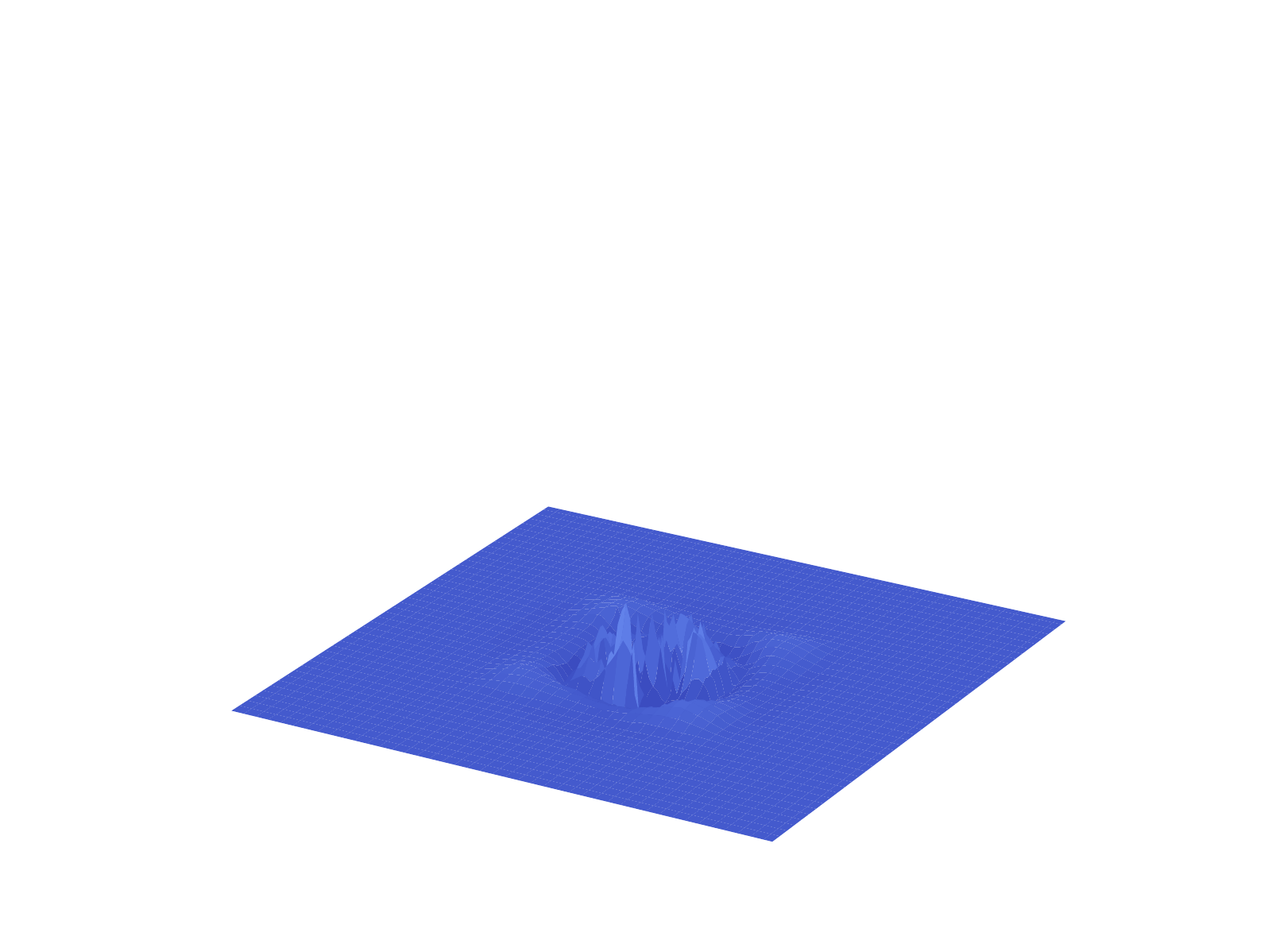}
		\caption{$\Rf\phi_i$.}
		\label{fig:corrbasis}
	\end{subfigure}
	\caption{The modified basis function $\phi_i - \Rf\phi_i$ and the Ritz-projected hat function $\Rf\phi_i$.}\label{fig:newbasis}
\end{figure}

At this point, one may construct a so-called ideal LOD method by simply replacing the standard finite element space $\VH$ by the multiscale space $\Vms$. However, the construction of $\Vms$ is based on the global projection~\eqref{eq:global_ritz_projection}, posed on the fine scale, which is neither feasible in terms of computational complexity nor in terms of memory. However, it is by now well-known (see~\cite{MalP14, HenP13}) that a basis correction $\Rf \phi_i$ decays exponentially fast away from its corresponding node. Therefore, it suffices to compute these corrections on local patches, consequently reducing the complexity without losing significant information. 

\subsection{Localization}

We begin by introducing the element-based grid patches to which the computations of $\{\Rf \phi_i\}_{i\in\NH}$ are to be restricted. With $N(K)$ as defined in~\eqref{eq:patch}, we define iteratively 
\begin{align}
	N^1(K) &:= N(K), \\
	N^\ell(K) &:= N(N^{\ell-1}(K)), \quad \ell \geq 2.
\end{align}
For an illustration of such patches for different choices of $\ell$, see Figure~\ref{fig:pathces}. Now, let $\VflK := \{w\in \Vf: \supp(w) \subseteq N^\ell(K)\}$, i.e., the space of fine-scale functions restricted to the local patch $N^\ell(K)$. Using this space, we can define a local Ritz-projection $\RflK: \Vh \rightarrow \VflK$ such that $\RflK v \in \VflK$ solves
\begin{align}
	a(\RflK v, w) = a(v, w),\label{eq:localized_ritz_projection_f}
\end{align}
for all $w\in \VflK$. We now sum the local contributions over all elements to get the corresponding global version as
\begin{align}
	\Rfl v = \sum_{K\in \KH} \RflK v.
\end{align}
The localized multiscale space can now be defined as $\Vmsl := \VH - \Rfl \VH$, spanned by the corrected basis functions $\{\phi_i - \Rfl \phi_i\}_{i\in \NH}$.

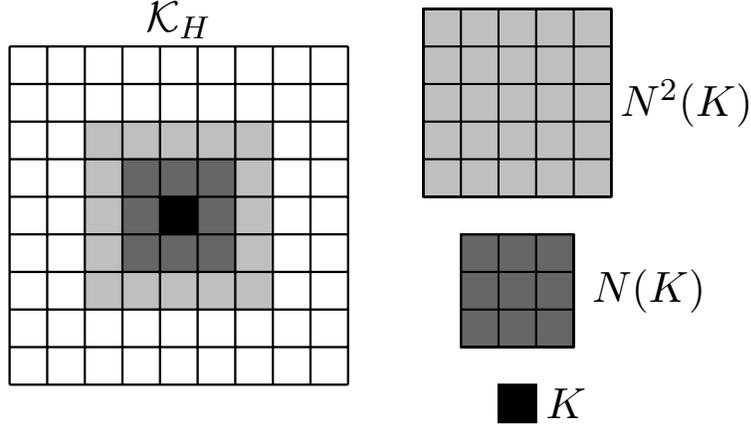
\begin{figure}
	\centering
	\begin{tikzpicture} [scale=0.5]
	
	\draw[fill=black!25!white] (2,2) rectangle (7,7);
	
	\draw[fill=black!60!white] (3,3) rectangle (6,6);
	
	\draw[fill=black!100!white] (4,4) rectangle (5,5);
	
	\draw [line width=0.3mm, draw=black, fill=black!20!white] (0,0) grid  (9,9);
	
	\draw[thick, fill=black!25!white] (11,5) grid (16,10) rectangle (11,5);
	
	\draw[thick, fill=black!60!white] (12,1) grid (15,4) rectangle (12,1);
	
	\draw[thick, fill=black!100!white] (13,-1) rectangle (14,0);
	
	\node[scale=1.5] at (18.0,7.5) {$N^2(K)$};
	
	\node[scale=1.5] at (17.0,2.5) {$N(K)$};
	
	\node[scale=1.5] at (14.8,-0.5) {$K$};
	
	\node[scale=1.5] at (4.5,9.7) {$\mathcal{K}_H$};
	
	\end{tikzpicture} 
	\caption{Illustration of coarse grid patches for different choices of the localization parameter $\ell$.}
	\label{fig:pathces}
\end{figure}

We want to express an equation similar to~\eqref{eq:fully_discrete_galerkin}, but defined on the space $\Vmsl$. For this, we introduce the $L_2(D)$-projection onto $\Vmsl$, $\Pmsl: \Vh \rightarrow \Vmsl$, such that for $v\in\Vh$, $\Pmsl v\in \Vmsl$ satisfies
\begin{align}
(\Pmsl v, w) =  (v,w),
\end{align}
for all $w\in\Vmsl$, and define the composed projection operator $\Pmslh : V\rightarrow \Vmsl$, by $\Pmslh := \Pmsl \circ \Ph$. One can show that such a projection is well-defined by expanding the functions in the basis $\{\phi_i - \Rfl \phi_i\}_{i\in \NH}$, which yields a symmetric, positive definite matrix system. We moreover define the localized diffusion operator $\Amsl : \Vmsl \rightarrow \Vmsl$ by the relation
\begin{align}
	(\Amsl v, w) = a(v,w), 
\end{align}
for all $v,w\in\Vmsl$. 

The proposed method now states: find $\Xnmsl \in \Vmsl$ such that
\begin{align}
	\Xnmsl - \Xnmslprev + k\Amsl \Xnmsl= \inttn{ \Pmslh G}
	\label{eq:proposed_method}
\end{align}
for $n=1,\ldots,N$, with initial value $\Xomsl= \Pmslh \Xo$. Similarly to the standard finite element Galerkin scheme, we can define $\Ekmsl{n} := ((I + k\Amsl)^{-1})^n$ and write the solution as
\begin{align}
	\Xnmsl = \Ekmsl{n} \Pmslh \Xo + \sum_{j=1}^n \inttj{\Ekmsl{n-j+1}\Pmslh G}.
\end{align}

\begin{remark}
	In~\cite{MalP20} an alternative formulation of LOD is proposed where $\Pmslh$ in the right-hand side of~\eqref{eq:proposed_method} is replaced by $\PH : V \rightarrow \VH$. This will simplify the computation of the right-hand side, at the cost of a more involved a priori error analysis.  
\end{remark}

\section{Error analysis}
\label{s:error_analysis}

We wish to derive a convergence result for the strong error between the reference solution on the fine scale $\Xhn$ from~\eqref{eq:fully_discrete_galerkin} and the multiscale approximation $\Xnmsl$ from~\eqref{eq:proposed_method}. We emphasize that such an error estimate implicitly yields strong convergence between $\Xnmsl$ and the exact solution $X(\tn)$ from~\eqref{eq:mild_solution}, since
\begin{align}
\|\Xnmsl - X(\tn)\|_{L_2(\Omega; L_2)} \leq \|\Xnmsl-\Xhn\|_{L_2(\Omega; L_2)} + \|\Xhn - X(\tn)\|_{L_2(\Omega; L_2)},
\end{align}
where the strong convergence rate for the second term is already known from Theorem~\ref{thm:strong_convergence_h}.

For the analysis, we make use of the Ritz-projection onto $\Vmsl$, defined as the operator $\Rmsl : \Vh \rightarrow \Vmsl$ satisfying
\begin{align}
a(\Rmsl v, w) = a(v, w), 
\label{eq:localized_ritz_projection}
\end{align}
for any $w\in\Vmsl$. The following well-known result for $\Rmsl$ will be used frequently. The proof is based on a standard Aubin--Nitsche duality argument, and can be found in, e.g.,~\cite{MalP18}. By $a\sim b$, it means that there exist constants $c$ and $C$ (independent of the mesh size and any variations in the diffusion) such that $cb\leq a\leq Cb$.

\begin{lemma}\label{lem:Rmsl_estimate}
	Let the localization parameter $\ell$ be chosen as $\ell\sim \log(1/H)$. Then, for any $v\in\Vh$, the localized Ritz-projection $\Rmsl$ in~\eqref{eq:localized_ritz_projection} satisfies
	\begin{align}
		\|v - \Rmsl v\| \leq CH|v|_1,
		\label{eq:localized_ritz_projection_bound}
	\end{align}
	where $C$ depends on the contrast in the diffusion coefficient $A$, but not on its oscillations. 
\end{lemma}

We split the error analysis into two parts. First, we discuss results related to deterministic parabolic problems. In a second part, we apply these results to the stochastic setting and prove strong convergence.

\subsection{Deterministic problems}

Let $\Un \in \Vh$ and $\Unmsl \in \Vmsl$ denote the solutions to~\eqref{eq:original_BE} and~\eqref{eq:proposed_method}, respectively, with zero right-hand side. That is,
\begin{align}
	\discd \Un + \Ah \Un&= 0, \\
	\discd \Unmsl + \Amsl \Unmsl &= 0,
\end{align}
where $\discd \Un = (\Un - \Unprev)/k$ denotes the backward discrete derivative. That is, the solutions can be written using the corresponding semigroups as $\Un= \Ekh^n \Ph \Xo$ and $\Unmsl = \Ekmsl{n}\Pmslh \Xo$. Moreover, define the operators $\Th := \Ah^{-1}\Ph$ and $\Tmsl := (\Amsl)^{-1}\Pmslh$, such that for $w_h\in\Vh$ and $w^{\mathrm{ms}}_\ell\in\Vmsl$, it holds that
\begin{align}
	a(\Th u, w_h) &= (u, w_h),\label{eq:Th}\\
	a(\Tmsl u, w^{\mathrm{ms}}_\ell) &= (u, w^{\mathrm{ms}}_\ell).\label{eq:Tmsl}
\end{align}
Then, $T_h$ and $\Tmsl$ can be viewed as solution operators to the following homogeneous equations in strong form
\begin{align}
T_h\discd \Un + \Un &= 0,\label{eq:Th_equation} \\
\Tmsl \discd \Unmsl + \Unmsl &= 0\label{eq:Tmsl_equation}.
\end{align}
Moreover, note that $\Tmsl = \Rmsl T_h$, since for $z \in \Vmsl \subset V_h$, we have
\begin{align}
a(\Tmsl u, z) &= a((\Amsl)^{-1}\Pmslh u, z) = (\Pmsl \Ph u, z) = (\Ph u,z) \\
&= a(\Ah^{-1}\Ph u, z) = a(\Rmsl \Ah^{-1}P_h u, z) = a(\Rmsl T_hu, z)
\end{align}
from which we can pass the term to the left-hand side and choose the test function $z = \Tmsl u - \Rmsl T_h u \in \Vmsl$ to prove the claim. At last, we require the following lemma for the error between the operators $\Th$ and $\Tmsl$.

\begin{lemma}\label{lem:T_estimate}
	Let $\Th$ be defined as in~\eqref{eq:Th} and $\Tmsl$ as in~\eqref{eq:Tmsl} with localization parameter chosen as $\ell\sim\log(1/H)$. Then, for $f\in L_2(D)$, it holds that
	\begin{align}
		\|(\Th - \Tmsl)f\| \leq CH^2\|f\|.
	\end{align}
\end{lemma}
\begin{proof}
	Consider the elliptic finite element problem to find $z_h \in \Vh$ such that
	\begin{align}
		a(z_h, w) = (f, w),\label{eq:elliptic_fe}
	\end{align}
	for all $w\in \Vh$. It holds that the solution is given by $z_h =\Th f$, since for $w\in \Vh$, we have by the definition of $\Th$ in~\eqref{eq:Th} that
	\begin{align}
		a(z_h, w) = a(\Th f, w) = (f, w).
	\end{align}
	Likewise, it holds that $z^{\mathrm{ms}}_\ell = \Tmsl f$ solves the elliptic LOD problem, i.e., the system~\eqref{eq:elliptic_fe} posed on the space $\Vmsl$. Therefore, by~\cite[Theorem~5.5]{MalP20}, it follows that
	\begin{align}
		\|(\Th - \Tmsl)f\| =  \|z_h - z^{\mathrm{ms}}_\ell\| \leq CH^2\|f\|,
	\end{align}
	which concludes the proof.
\end{proof}

 We are now ready to prove the following properties, which are crucial for the proof of strong convergence later on.

\begin{lemma}\label{lem:F_properties}
	Let $\Fmsl{n} := \Ekmsl{n}\Pmslh - \Ekh^n \Ph$. Then, for $\ord\in[1,2]$, the properties
	\begin{align}
	\Big( k \sum_{j=1}^n \|\Fmsl{j}v\|^2 \Big)^{1/2} &\leq CH^\ord|v|_{\ord-1},\label{eq:semigroup_property_1} \\
	\|\Fmsl{n}v\| &\leq CH^\ord\tn^{-1/2}|v|_{\ord-1}, \label{eq:semigroup_property_2}
	\end{align}
	hold, where $C$ depends on the contrast in the diffusion $A$, but not on its variations.
\end{lemma}
\begin{proof}
	Define the error $\en := \Fmsl{n}v = \Unmsl - \Un$, and note that it satisfies the equation
	\begin{align}
	\begin{split}
	\Tmsl\discd \en + \en &= \Tmsl \discd \Unmsl - \Tmsl \discd \Un + \Unmsl - \Un \\
	&= (\Th - \Tmsl)\discd \Un= (I-\Rmsl) \Th \discd \Un = (\Rmsl - I)\Un=: \rho^n.
	\end{split}\label{eq:first_rho_equation}
	\end{align}
	Here, we used~\eqref{eq:Th_equation}--\eqref{eq:Tmsl_equation} in the second equality, and in the third step the fact that $\Tmsl=\Rmsl\Th$. Now, we test this equation with $\en$ and apply Cauchy--Schwarz and Young's inequality. This yields
	\begin{align}
	(\Tmsl \discd \en, \en) + (\en, \en) = (\rho^n, \en) \leq \frac{1}{2}\|\rho^n\|^2 + \frac{1}{2}\|\en\|^2,
	\end{align}
	or, by passing the last term to the left-hand side, we get
	\begin{align}
		(\Tmsl\discd\en,\en) + \frac{1}{2}\|\en\|^2 \leq \frac{1}{2}\|\rhon\|^2.
		\label{eq:rho_equation}
	\end{align}
	Note that we can use the definition of $\Tmsl$ to create a lower bound for the first term, as
	\begin{align}
	(\Tmsl \discd \en, \en) = (\en, \Tmsl \discd \en) = a(\Tmsl \en, \Tmsl \discd \en) \geq \frac{1}{2}\discd |\Tmsl \en|^2_1,
	\end{align}
	where the final inequality follows since, for any temporally discrete function $f^n$, we have
	\begin{align}
	\begin{split}
	(f^n, \discd f^n) &= \tfrac{2}{k}\big( \|f^n\|^2 - (f^n, f^{n-1}) \big) \\
	&\geq \tfrac{2}{k} \big( \|f^n\|^2 - \big(\tfrac{1}{2}\|f^n\|^2 + \tfrac{1}{2}\|f^{n-1}\|^2 \big) \big) \\
	&= \tfrac{1}{k}\big(\|f^n\|^2 - \|f^{n-1}\|^2 \big) = \discd\|f^n\|^2.
	\end{split}\label{eq:disc_temp_function}
	\end{align}
	Therefore, by multiplying~\eqref{eq:rho_equation} with the time step $k$, and summing over $j=1,\ldots,n$, we get
	\begin{align}
	|\Tmsl \en|^2_1 + k \sum_{j=1}^n \|\ej\|^2 \leq k\sum_{j=1}^n \|\rho^j\|^2.
	\label{eq:tmp}
	\end{align}
	By interpolation theory, it suffices to show the properties~\eqref{eq:semigroup_property_1}--\eqref{eq:semigroup_property_2} for $\ord=1$ and $\ord=2$. 
	
	The case $\ord=2$ follows from~\eqref{eq:tmp}, since
	\begin{align}
	k\sum_{j=1}^n \|\ej\|^2 \leq k\sum_{j=1}^n \|\rho^j\|^2 \leq CH^4k\sum_{j=1}^n \|\discd \Uj\|^2 \leq CH^4|v|_1^2,
	\end{align}
	where the second inequality follows from Lemma~\ref{lem:T_estimate}. 
	
	For $\ord=1$, we make a similar estimate, but apply the Ritz-projection instead, i.e.,
	\begin{align}
		k\sum_{j=1}^n \|\rhoj\|^2 = k\sum_{j=1}^n \|(\Rmsl-I)\Un\|^2 \leq CH^2k\sum_{j=1}^n|\Uj|_1^2 \leq CH^2\|v\|,
	\end{align}	
	where the first inequality follows from Lemma~\ref{lem:Rmsl_estimate}.
	
 For~\eqref{eq:semigroup_property_2}, we test~\eqref{eq:first_rho_equation} with $\discd \en$ and follow the calculations in~\cite[Lemma~4.2 \& 4.3]{MalP18}, to show that
	\begin{align}\label{eq:en_estimate}
		\|\en\| \leq C\tn^{-1}\bigg( \max_{2\leq j\leq n} \tj^2\|\discd\rhoj\| + \max_{1\leq j\leq n} \Big( \tj\|\rhoj\| + \Big\| k\sum_{r=1}^{j} \rho_r\Big\| \Big) \bigg).
	\end{align}
	For the first term, we once again apply Lemma~\ref{lem:T_estimate} and note that
	\begin{align}
		\tj^2\|\discd\rhoj\| = \tj^2\|(\Th - \Tmsl)\discd^2\Uj\| \leq CH^2\tj^2\|\discd^2\Uj\| \leq CH^2\tj^{1/2}|v|_1,
	\end{align}
	where the last inequality follows from Lemma~\ref{lem:Ekh_estimate}. By applying similar calculations for the second term, we get
	\begin{align}
		\tj\|\rhoj\| = \tj\|(\Th - \Tmsl)\discd\Uj\| \leq CH^2\tj\|\discd\Uj\| \leq CH^2\tj^{1/2}|v|_1.
	\end{align}
	For the last term, we note that
	\begin{align}
		\max_{1\leq j\leq n} \Big\| k\sum_{r=1}^{j} \rho_r\Big\| \leq \max_{1\leq j\leq n} k\sum_{r=1}^{j} \|\rhoj\| = \max_{1\leq j\leq n} \tj\|\rhoj\| \leq \max_{1\leq j\leq n} CH^2\tj^{1/2}|v|_1.
	\end{align}
	In total, inequality~\eqref{eq:en_estimate} becomes
	\begin{align}
		\|\en\| \leq CH^2\tn^{-1/2}|v|_1,
	\end{align}
	which shows~\eqref{eq:semigroup_property_2} for $\ord=2$. 
	
	For $\ord=1$, it follows from~\cite[Lemma~4.2]{MalP18} that
	\begin{align}
		\|\en\| \leq C\Big( \max_{2\leq j\leq n} \tj \|\discd \rhoj\| + \max_{1\leq j\leq n} \|\rhoj\| \Big)
	\end{align}
	for $n\geq 2$. Using the Ritz-projection, the two terms can be bounded by
	\begin{align}
	\|\rhoj\| &= \|\Rmsl \Uj - \Uj\| \leq CH |\Uj|_1 \leq CH\tj^{-1/2}\|v\|,\\
	\tj\|\discd \rhoj\| &= \tj \|\Rmsl \discd \Uj - \discd \Uj\| \leq CH\tj|\discd \Uj|_1 \leq CH\tj^{-1/2}\|v\|,
	\end{align}
	where in the last steps we applied Lemma~\ref{lem:Ekh_estimate}.
\end{proof}

\subsection{Strong convergence}

We are now ready to prove the strong convergence between our multiscale method and the reference solution. 

\begin{theorem}\label{thm:strong_convergence_lod}
	Let $\Xnmsl \in \Vmsl$ be the solution to the multiscale method defined in~\eqref{eq:proposed_method}, and let $\Xhn\in \Vh$ be the reference solution from~\eqref{eq:original_BE}. Let $\ord\in[1,2]$, and assume that $\|\Lambda^{(\ord-1)/2}G\|_{L_2^0} < +\infty$ and $\Xo \in L_2(\Omega; \dot{H}^{\ord-1})$. Then, the strong error between $\Xnmsl$ and $\Xhn$ satisfies
	\begin{align}
	\|\Xnmsl - \Xhn\|_{L_2(\Omega; L_2)} \leq CH^\ord\big(\tn^{-1/2} \|\Xo\|_{L_2(\Omega,\dot{H}^{\ord-1})} + \|\Lambda^{(\ord-1)/2}G\|_{L_2^0} \big),
	\end{align}
	where the constant $C$ is independent of the variations in the diffusion $A$.
\end{theorem}

\begin{proof}
	Define $\en:= \Xnmsl - \Xhn$ and $\Fmsl{n} := \Ekmsl{n}\Pmslh - \Ekh^n \Ph$. Then, we can write the error as
	\begin{align}
	\en= \Fmsl{n}\Xo + \sum_{j=1}^n \inttj{\Fmsl{n-j+1}G} =: e_1^n + e_2^n.
	\end{align}
	The first error term can be bounded using the property \eqref{eq:semigroup_property_2} from Lemma~\ref{lem:F_properties}, as
	\begin{align}
	\|e_1^n\| = \|\Fmsl{n}\Xo\| \leq CH^\ord\tn^{-1/2}|\Xo|_{\ord-1}.
	\end{align}
	For the second error, we utilize the Itô isometry, and get
	\begin{align}
	\|e_2^n\|_{L_2(\Omega; L_2)}^2 &= \EBig{\Big\| \sum_{j=1}^n \inttj{\Fmsl{n-j+1}G} \Big\|^2} \\
	&= \sum_{j=1}^n \EBig{\Big\| \inttj{\Fmsl{n-j+1}G} \Big\|^2} \\
	& = \sum_{j=1}^n \int_{\tjprev}^{\tj} \|\Fmsl{n-j+1}G\|^2_{L_2^0}\, \ds \\
	&= k\sum_{j=1}^n \|\Fmsl{n-j+1}G\|^2_{L_2^0} \\
	&= k\sum_{j=1}^n \sum_{\ell=1}^\infty \|\Fmsl{n-j+1}GQ^{1/2}\varphi_\ell\|^2 \\
	&\leq CH^{2\ord} \sum_{\ell=1}^\infty \|\Lambda^{(\ord-1)/2}GQ^{1/2}\varphi_\ell\|^2 = CH^{2\ord}\|\Lambda^{(\ord-1)/2}G\|_{L_2^0}^2,
	\end{align}
	where the first equality holds since the mixed terms have zero mean, and the last inequality follows from the property~\eqref{eq:semigroup_property_1} in Lemma~\ref{lem:F_properties}. This concludes the proof.
\end{proof}

\section{Weak error}\label{s:weak_error}

In the previous section, we considered the convergence of the error in mean square for the proposed method~\eqref{eq:proposed_method}. However, in many cases it is preferable to analyze the expectation of the solution, or more generally some function of the solution often called a \emph{quantity of interest}. For this purpose, let $g:L_2(D)\rightarrow\calB$ be a Lipschitz continuous function with values in a separable Hilbert space $\calB$. In this section, we focus on the error between $\E{g(\Xhn)}$ and its multiscale approximation $\E{g(\Xnmsl)}$, referred to as the \textit{weak error}, given by
\begin{align}
	\|\E{g(\Xhn)} - \E{g(\Xnmsl)}\|_\calB.
\end{align}

It is a well-known fact that the weak error is bounded from above by the strong error. Indeed, we immediately have
\begin{align}
\|\E{g(\Xhn)} - \E{g(\Xnmsl)}\|^2_\calB 
&\leq \E{\|g(\Xhn) - g(\Xnmsl)\|^2_\calB} \\
&\leq L_g^2\E{\|\Xhn - \Xnmsl\|^2} \\
&= L_g^2\|\Xhn - \Xnmsl\|_{L_2(\Omega;L_2)}^2,
\end{align}
where we first used Jensen's inequality, followed by the fact that $g$ is Lipschitz continuous with constant $L_g$. Consequently, we can apply the strong convergence rate from Theorem~\ref{thm:strong_convergence_lod}, and the weak error satisfies
\begin{align}
	\|\E{g(\Xhn)} - \E{g(\Xnmsl)}\|_\calB \leq CH^2\Big(\tn^{-1/2}\|\Xo\|_{L_2(\Omega; \dot{H}^1)} + \|\Lambda^{1/2}G\|_{L^0_2}\Big).\label{eq:weak_error_h}
\end{align}

\begin{remark}
	It is a well-established fact that the discretization error for SPDEs generally achieves twice the order of weak convergence in comparison to the strong order of convergence (see, e.g.,~\cite{AndKL16, KovLL10, KovLL12}). However, with sufficient regularity on the initial value $\Xo$ and the covariance operator $Q$, the strong error bound of an equation with additive noise yields quadratic convergence in space, which corresponds to the optimal rate for piecewise linear polynomials. In this paper, the emphasis lies on highlighting the computational efficiency of applying LOD to approximate the solution $\E{g(\Xhn)}$. Therefore, we limit the scope of the weak error analysis to the case of sufficient regularity, and leave the more thorough details of lower regularity assumptions to future studies.
\end{remark}

To measure the weak error, we are required to compute the expectation of $g(\Xnmsl)$. However, since the solution $\Xnmsl$ to~\eqref{eq:proposed_method} depends on the realization of the noise $W$, the behavior of $\E{g(\Xnmsl)}$ is not obvious from one single simulation. Instead, the expectation is approximated by a Monte-Carlo estimator, which averages the solution over a large number of samples. We consider two such approaches. First, we apply the standard Monte-Carlo estimation technique. Subsequently, we reduce the computational complexity by considering a multilevel Monte-Carlo estimator.

\subsection{Monte-Carlo estimation}

The expected value is approximated by the standard Monte-Carlo estimator as
\begin{align}
	\E{g(\Xnmsl)} \approx E_M[g(\Xnmsl)] := \frac{1}{M}\sum_{m=1}^M g(\Xnmslm),\label{eq:monte_carlo_estimator}
\end{align}
where $\{\Xnmslm\}_{m=1}^M$ are independent and identically distributed random variables with the same distribution as $\Xnmsl$.

\sloppy We continue by considering the error $e^n_M := \E{g(\Xhn)} - E_M[g(\Xnmsl)]$. First of all, note that the error can be decomposed as
\begin{align}
	&\|e^n_M\|_{L_2(\Omega;\calB)} \leq \|\E{g(\Xhn)} - \E{g(\Xnmsl)}\|_\calB + \|\E{g(\Xnmsl)} - E_M[g(\Xnmsl)]\|_{L_2(\Omega;\calB)},\label{eq:error_decomposition}
\end{align}
where the first term corresponds to the weak discretization error using the proposed LOD method and the second term represents the statistical error from the Monte-Carlo sampling. The discretization error is bounded by~\eqref{eq:weak_error_h}. For the statistical error, we have the following equality (see~\cite[Lemma~4.1]{BarLS13})
\begin{align}
	\|\E{g(\Xnmsl)} - E_M[g(\Xnmsl)]\|_{L_2(\Omega;\calB)} = \frac{\text{Var}[g(\Xnmsl)]^{1/2}}{\sqrt{M}}.
\end{align}
The error between the expectation $\E{g(\Xhn)}$ and the Monte-Carlo LOD approximation $E_M[g(\Xnmsl)]$ is thus bounded by
\begin{align}
	\| \E{g(\Xhn)} - E_M[g(\Xnmsl)]\|_{L_2(\Omega;\calB)} \leq C\Big(H^2  + \frac{1}{\sqrt{M}}\Big).\label{eq:MC_rate}
\end{align}
In order to maintain the quadratic convergence, it is thus required to choose a sample size proportional to the coarse mesh size as $M \sim H^{-4}$.

\subsection{Multilevel Monte-Carlo}
\label{s:multilevel}

Next, we consider the multilevel Monte-Carlo approach for the approximation of $\E{g(\Xnmsl)}$. The technique was introduced for stochastic differential equations in~\cite{Gil08} and first used in the context of SPDEs in~\cite{BarLS13}. It is based on writing the expectation as a telescopic sum over all spatial discretization levels. For this purpose, denote by $H_j$ the mesh size of the discretization $\mathcal{K}_{H_j}$, and let $\XnmslHj$ be the solution to~\eqref{eq:proposed_method} with $\Vmsl$ based on the discretization $\mathcal{K}_{H_j}$. Then, the expectation can be written as the telescopic sum
\begin{align}
\E{g(\XnmslHj)} = \E{g(\XomslHj)} + \sum_{j=1}^{J} \E{g(\XnmslHj) - g(\XnmslHjprev)}.
\end{align}
In this way, each term in the sum can be estimated by a Monte-Carlo estimator with different sample size, such that the multilevel Monte-Carlo estimator is given by
\begin{align}
E^J[g(\XnmslHJ)] &:= E_{\Mtildelow_0}[g(\XomslHj)] + \sum_{j=1}^{J} E_{\Mtildelow_j}[g(\XnmslHj) - g(\XnmslHjprev)] \label{eq:mlmc_estimator}\\
&= \frac{1}{\Mtilde_0}\sum_{m=1}^{\Mtildemid_0} g(\XomslHjm) + \sum_{j=1}^{J} \frac{1}{\Mtilde_j} \sum_{m=1}^{\Mtildemid_j} g(\XnmslHjm) - g(\XnmslHjprevm),
\end{align}
where $\Mtilde_j$ denotes the number of samples used on the discretization level $\mathcal{K}_{H_j}$ and $E_{\Mtildelow_j}[\cdot]$ is defined as in~\eqref{eq:monte_carlo_estimator}.

The idea behind the multilevel Monte-Carlo approach is to choose the number of samples $\{\Mtilde_j\}_{j=0}^J$ such that the majority of the samples are allocated on the coarser levels, while the weak convergence remains the same as for the standard Monte-Carlo estimator. By repeating the calculations in~\cite[Corollary~3.8]{BarL12}, it follows that the multilevel estimator satisfies
\begin{align}
\|\E{g(X^n_h)} - E^J[g(\XnmslHJ)]\|_{L_2(\Omega;\calB)} \leq C\Bigg( H^2_J + \frac{1}{\sqrt{\Mtilde_0}} + \sum_{j=1}^J \frac{H_j^2}{\sqrt{\Mtilde_j}}\Bigg).
\end{align}
We want to choose $\{\Mtilde_j\}_{j=0}^J$ such that each term is proportional to $H_J^2$. Therefore, we can choose $\Mtilde_0\ = \gamma H_J^{-4}$ where $\gamma$ is a proportionality constant, and $\Mtilde_j\ =\ \Mtilde_0H_j^4\cdot 2^{2\delta j}$ for some $\delta > 0$. Then, the first two terms satisfy $\mathcal{O}(H_J^2)$, and for the final sum we note that
\begin{align}
\sum_{j=1}^J \frac{H_j^2}{\sqrt{\Mtilde_j}} = \sum_{j=1}^J \frac{1}{\sqrt{\Mtilde_0\cdot 2^{2\delta j}}} = \frac{1}{\sqrt{\Mtilde_0}}\sum_{j=1}^J 2^{-\delta j} \leq \gamma^{-1/2}H_J^2 \frac{1}{2^\delta - 1},
\end{align}
where the final inequality follows by bounding the geometric sum by its limit as $J\rightarrow \infty$. Note that choosing a large $\delta$ gives a better constant in the error estimate, but makes the number of samples $\{\Mtilde_j\}_{j=0}^J$ decay slower, and vice-versa for a small $\delta$. That is, $\delta$ controls the trade-off between the constant in the error estimate and the computational complexity of the method.

\section{Numerical examples}
\label{s:numerical_examples}

In this section we illustrate the performance of the main method~\eqref{eq:proposed_method} in two parts. In Subsection~\ref{s:strong_convergence}, we verify the strong convergence derived in Theorem~\ref{thm:strong_convergence_lod} between the LOD method and the reference FEM solution computed on a fine grid. Next, in Subsection~\ref{s:weak_convergence}, we analyze the weak convergence of the LOD solution. The weak error is analyzed in the sense of the expectation of the solution, i.e., the function $g$ from Section~\ref{s:weak_error} is set to the identity mapping. This is done for simplicity, as the exact solution satisfies a deterministic PDE, and therefore we circumvent the computationally heavy Monte-Carlo estimation of a reference solution. The expectation is estimated by combining the LOD method with a Monte-Carlo estimator and a multilevel Monte-Carlo estimator, respectively. Finally, in Subsection~\ref{s:computational_complexity}, we close with a discussion on the computational complexity of the different methods.

As model problem, we consider the system~\eqref{eq:spde} with an additional source term $f$ added to the right-hand side, i.e.,
\begin{align}
\dX(t) + \Lambda X(t)\, \dt = f(t)\, \dt + G\, \dW(t).
\label{eq:weak_model_problem}
\end{align}
By adding this source function, the amplitude of the solution will not decay towards zero equally fast. In turn, this allows for a larger contrast in the diffusion, making the resulting solution more interesting. 
\begin{remark}
	In this paper, we have considered an equation of the form~\eqref{eq:weak_model_problem} with $f\equiv0$. However, since the problem is linear, one can split the solution as $X = X_1 + X_2$, where $X_1$ solves~\eqref{eq:weak_model_problem} with $f\equiv 0$, and $X_2$ is the solution with vanishing noise, i.e., the deterministic case where $W \equiv 0$. The error can then be split as
	\begin{align}
	\|X^n_{h} - X^{n}_{\mathrm{ms},\ell}\|_{L_2(\Omega; L_2)} \leq \|X^n_{h,1} - X^{n}_{\mathrm{ms},\ell,1}\|_{L_2(\Omega; L_2)} + \|X^n_{h,2} - X^{n}_{\mathrm{ms},\ell,2}\|_{L_2(\Omega; L_2)},
	\end{align}
	where the second term is strictly deterministic, and the analysis follows directly from the work in~\cite{MalP18}.
\end{remark}

The data is similar in all numerical examples. In space, we consider a unit square domain, i.e., $D = [0,1] \times [0,1]$. The fine discretization is done with mesh size $h=2^{-8}$, while we vary the coarse mesh size $H$ for the convergence analysis. The temporal domain~$[0,T]$ is discretized using a uniform time step $k=0.01$ and a final time $T=0.5$. The diffusion coefficient $A(x,y)$ is a piecewise constant function whose values vary on a scale of $\varepsilon = 2^{-6}$, such that $h<\varepsilon$ clearly resolves the variations. We remark that the coefficient is generated once and remains the same through every realization of the equation. As initial data, we set $\Xo(x,y) = \sin(\pi x)\sin(\pi y)$, and the source function is constant $f(x,y,t) = 5$. An illustration of the diffusion coefficient and the initial data can be seen in Figure~\ref{fig:data}.

\begin{figure}
	\centering
	\begin{subfigure}[b]{0.44\textwidth}
		\includegraphics[width=\textwidth]{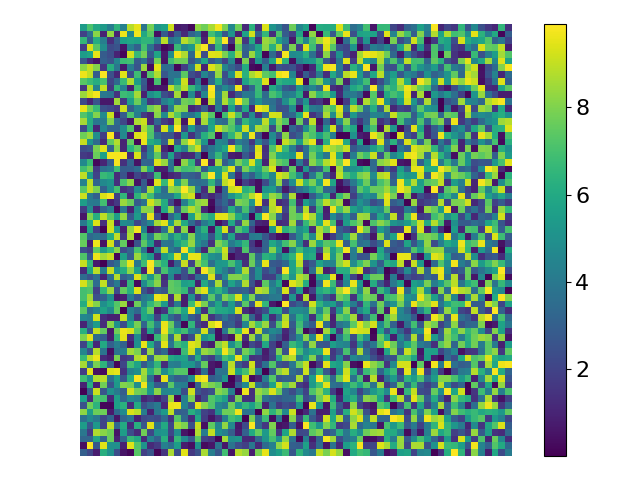}
		\caption{\small $A(x,y)$.}
	\end{subfigure}
	~ 
	\begin{subfigure}[b]{0.44\textwidth}
		\includegraphics[width=\textwidth]{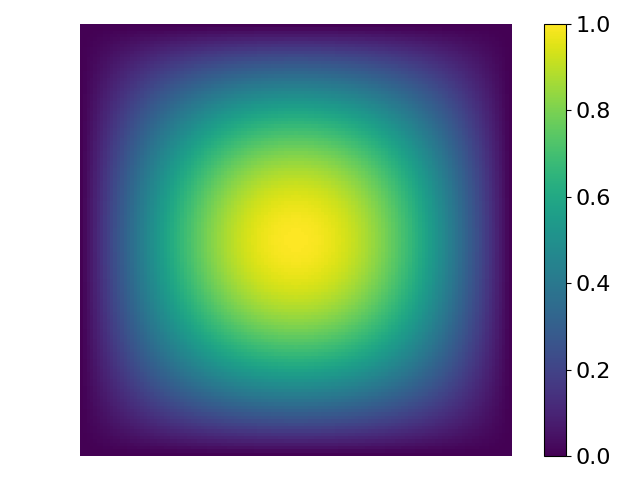}
		\caption{\small $\Xo(x,y)$.}
	\end{subfigure}
	\caption{\small Diffusion coefficient $A(x,y)$ and initial value $\Xo(x,y)$ used for the numerical examples.
	}\label{fig:data}
\end{figure}

For the implementation, it is necessary to construct a noise approximation of the $Q$-Wiener process $W$. First of all, we define the trace class operator $Q$ through the relation $Qe_{m,n} = \lambda_{m,n}e_{m,n}$, where we set the eigenfunctions to $e_{m,n} = \sin(m\pi x)\sin(n\pi y)$, with corresponding eigenvalues $\lambda_{m,n} = \Theta (m^{2+\theta}+n^{2+\theta})^{-1}$. Here, the parameters $\Theta$ and $\theta$ determine the amplitude and decay rate of the noise, respectively, and consequently affect the variance and smoothness of the corresponding solution. Moreover, recall the Karhunen--Loève expansion of $W$ in~\eqref{eq:Karhunen-Loeve}. For the numerical simulations, the common approach is to truncate such an expansion up to a truncation parameter $\kappa$. That is, the noise we consider for the examples is given by
\begin{align}
	W^\kappa(x,y,t) = \sum_{m=1}^{\kappa}\sum_{n=1}^{\kappa} \sqrt{\lambda_{m,n}}\beta_{m,n}(t)e_{m,n}(x,y),
\end{align}
where $\{\beta_{m,n}(t)\}_{m,n}$ are mutually independent, real-valued Brownian motions. The truncation parameter is set to $\kappa \sim h^{-1}$ for all examples. This particular choice is common for the noise approximation in SPDE simulations and will not dominate the error if the eigenvalues of $Q$ decay sufficiently fast (see, e.g.,~\cite{BarL12b, KovLL10}).

\subsection{Strong convergence}
\label{s:strong_convergence}

In this part, we illustrate numerically the strong convergence derived in Theorem~\ref{thm:strong_convergence_lod}. For the noise, we have set the eigenvalues to $Q$ as $\lambda_{m,n} = (m^{2+0.01}+n^{2+0.01})^{-1}$, and the truncation parameter to $\kappa=h^{-1}/16$. For an illustration of one realization of the noise $W^\kappa(t)$ and the corresponding reference solution $\X^n_{h}$ at final time $T=0.5$, see Figure~\ref{fig:sol_noise}. 

The LOD solution is computed for coarse spatial mesh sizes $H = 2^{-i}$, $i=1,\ldots, 6$, with localization parameter $\ell= \log_2(1/H)$. We compare the error in $L_2(\Omega; L_2)$-norm at final time $T=0.5$ against a reference solution computed with standard FEM on the fine grid. The norm is estimated by a Monte-Carlo estimator with 100 samples, i.e., $\|\cdot\|_{L_2(\Omega;L_2)} \approx E_{100}[\|\cdot\|]$. For comparison, the convergence rate when computing the coarse solution using standard FEM is also included, as well as an $\mathcal{O}(H^2)$-reference line which indicates the sought convergence rate from Theorem~\ref{thm:strong_convergence_lod}. In the convergence plot in Figure~\ref{fig:strong_error}, it is shown that the LOD solution achieves quadratic convergence, as predicted by Theorem~\ref{thm:strong_convergence_lod}, while the error for the FEM solution remains constant through all mesh sizes as it is unable to resolve the variations of $A$.

\begin{figure}
	\centering
	\begin{subfigure}[b]{0.44\textwidth}
		\includegraphics[width=\textwidth]{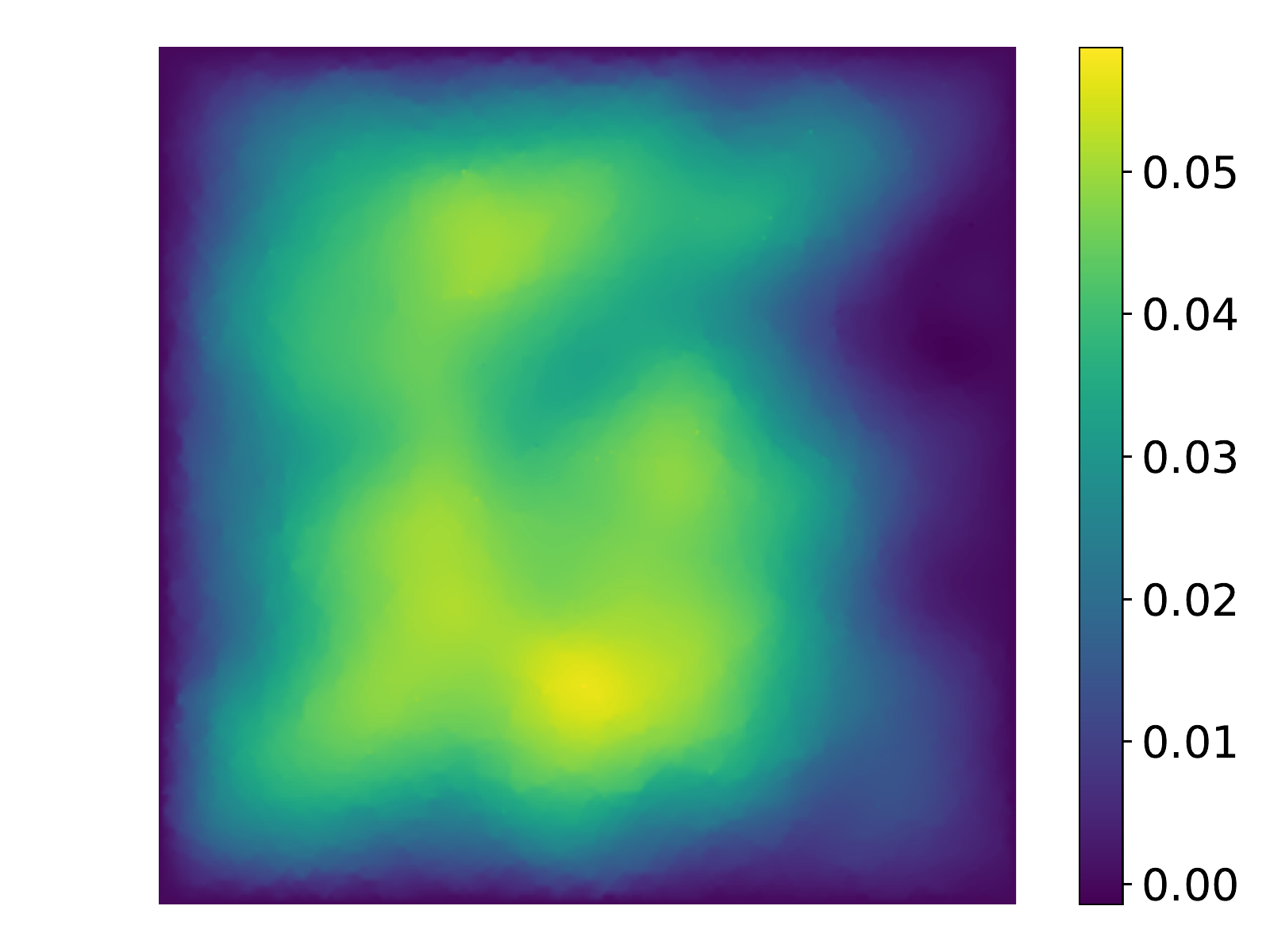}
		\caption{\small $X^N_{h}$.}
	\end{subfigure}
	~ 
	\begin{subfigure}[b]{0.44\textwidth}
		\includegraphics[width=\textwidth]{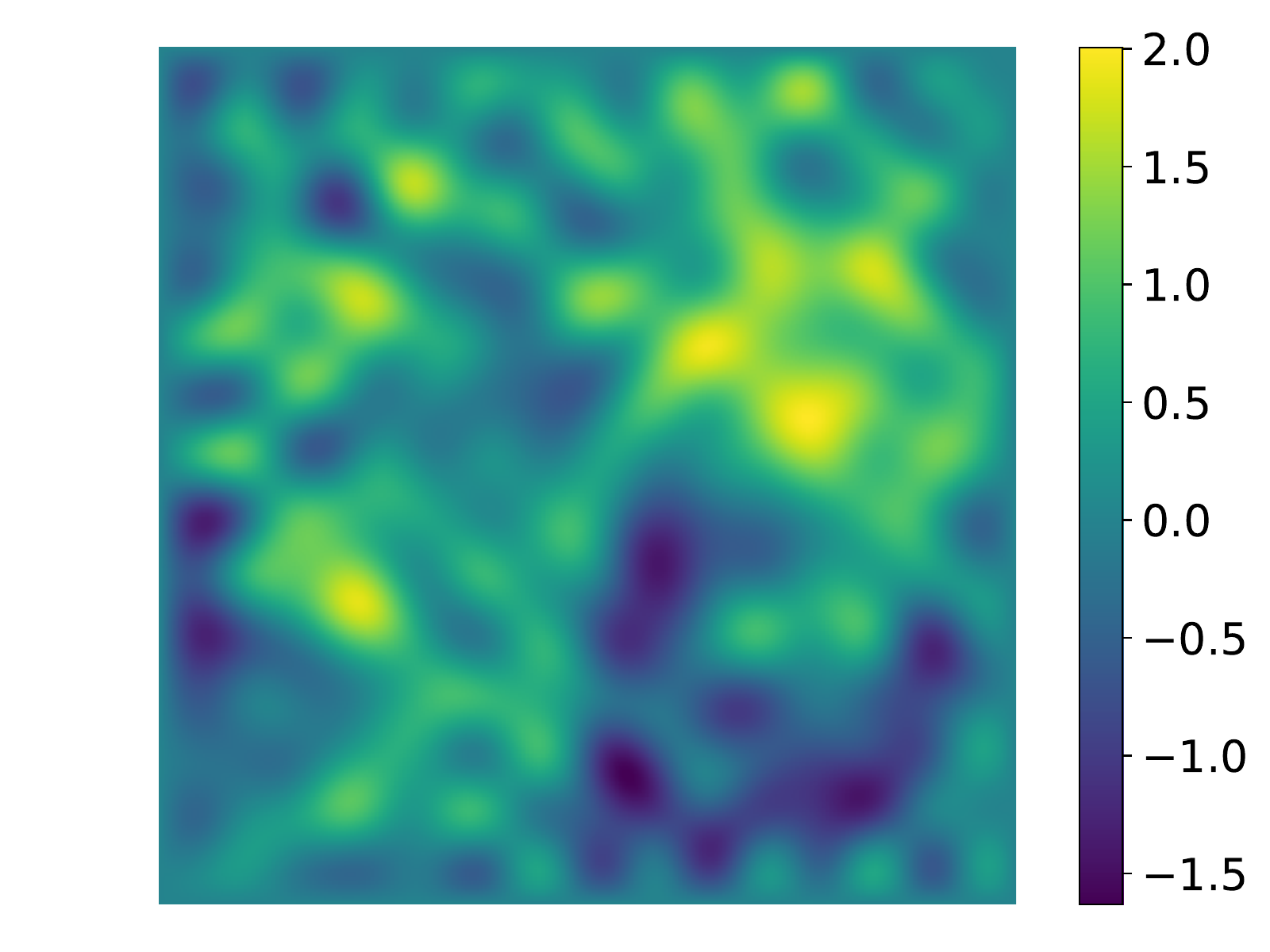}
		\caption{\small $W^\kappa(x,y,t_N)$.}
	\end{subfigure}
	\caption{\small The reference solution and the noise, plotted at final time $t_N=0.5$.
	}\label{fig:sol_noise}
\end{figure}

\begin{figure}
	\centering
	\begin{tikzpicture}
	
	\begin{axis}[
	scale=0.9,
	width=4.602in,
	height=3.1in,
	legend cell align={left},
	legend style={
		fill opacity=0.8,
		draw opacity=1,
		text opacity=1,
		at={(0.97,0.03)},
		anchor=south east,
		draw=lightgray204
	},
	log basis x={2},
	log basis y={2},
	tick align=outside,
	tick pos=left,
	x grid style={darkgray176},
	xlabel={\(\displaystyle H\)},
	xmajorgrids,
	xmin=0.0131390064883393, xmax=0.594603557501361,
	xmode=log,
	xtick style={color=black},
	y grid style={darkgray176},
	ylabel={\(\displaystyle E_{100}[\|X^N_{h} - X^{N}_{\mathrm{ms},\ell}\|]\)},
	ymajorgrids,
	ymin=0.00267162950801003, ymax=9.86845924751361,
	ymode=log,
	ytick style={color=black}
	]
	\addplot [semithick, steelblue31119180, mark=x, mark size=3, mark options={solid}]
	table {%
		0.5 5.38758152
		0.25 1.4472127
		0.125 0.363989834
		0.0625 0.0755056519
		0.03125 0.0126911928
		0.015625 0.00388091473
	};
	\addlegendentry{LOD}
	\addplot [semithick, forestgreen4416044, mark=*, mark size=3, mark options={solid}]
	table {%
		0.5 6.7934672
		0.25 3.71578719
		0.125 2.88766741
		0.0625 2.69144673
		0.03125 2.56982106
		0.015625 2.18550044
	};
	\addlegendentry{FEM}
	\addplot [semithick, sienna1408675, opacity=0.4, dashed]
	table {%
		0.5 5
		0.25 2.5
		0.125 1.25
		0.0625 0.625
		0.03125 0.3125
		0.015625 0.15625
	};
	\addlegendentry{$\mathcal{O}(H)$}
	\addplot [semithick, crimson2143940, opacity=0.4, dashed]
	table {%
		0.5 5
		0.25 1.25
		0.125 0.3125
		0.0625 0.078125
		0.03125 0.01953125
		0.015625 0.0048828125
	};
	\addlegendentry{$\mathcal{O}(H^2)$}
	\end{axis}
	
	\end{tikzpicture}
	
	\caption{The strong error at final time $t_N=0.5$ between the multiscale approximation $X^{N}_{\mathrm{ms}, \ell}$ and the reference solution $X^N_{h}$ computed on the fine grid. The norm is estimated as $\|\cdot\|_{L_2(\Omega;L_2)} \approx E_{100}[\|\cdot\|]$, i.e., by a Monte-Carlo estimator with $M=100$ samples.\label{fig:strong_error}}
\end{figure}
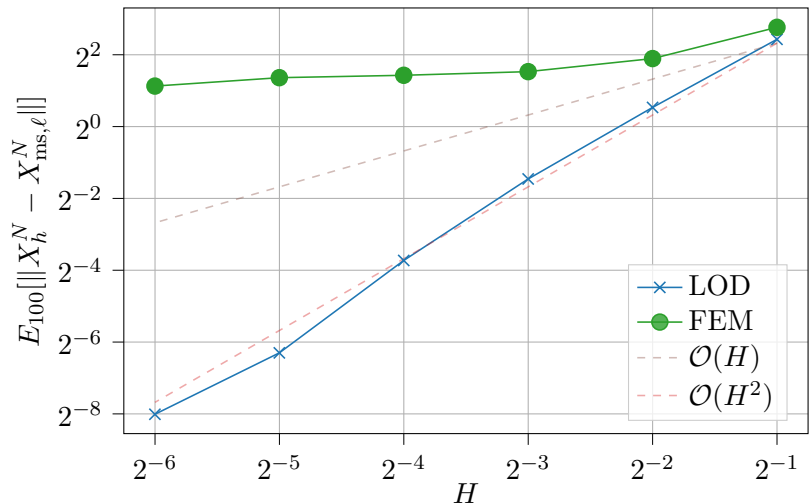

\subsection{Weak convergence}
\label{s:weak_convergence}

We study the convergence of the weak error between the expectation of the reference solution $\E{\Xhn}$ and its corresponding multiscale approximation. In order to compare the performance between the standard Monte-Carlo and the multilevel Monte-Carlo estimators, we fix $H = H_J$ for the standard Monte-Carlo estimation and denote by $X^n_{\mathrm{ms},\ell,J}$ the LOD approximation on the coarse grid with mesh size $H_J$. The term $\E{X^n_{\mathrm{ms}, \ell, J}}$ is estimated by applying a standard Monte-Carlo estimator and a multilevel Monte-Carlo estimator, respectively, in combination with the LOD method. To simplify the computations, we lower the variance from the previous example by choosing $\lambda_{m,n} = 5^{-2}(m^{2+0.01} + n^{2+0.01})^{-1}$. 
\begin{remark}
	The expectation of the solution to~\eqref{eq:weak_model_problem} does not in practice require an estimation via Monte-Carlo simulation. Due to the linearity, it holds that the expectation $\E{X}$ satisfies the deterministic equation
	\begin{align}
		\d\E{X(t)} + \Lambda\E{X(t)}\, \dt = f(t)\, \dt.
		\label{eq:E_equation}
	\end{align}
	However, as shown in Section~\ref{s:weak_error}, the Monte-Carlo estimator can be used to approximate any Lipschitz continuous function of the solution to~\eqref{eq:weak_model_problem}. In the example, we use the identity $g(X^n_{\mathrm{ms},\ell,J}) = X^n_{\mathrm{ms},\ell,J}$, since the reference solution can be computed using~\eqref{eq:E_equation}, which removes the issue of requiring a Monte-Carlo estimation for the reference solution.
\end{remark}

\begin{figure}
	\centering
	\includegraphics[width=0.9\textwidth]{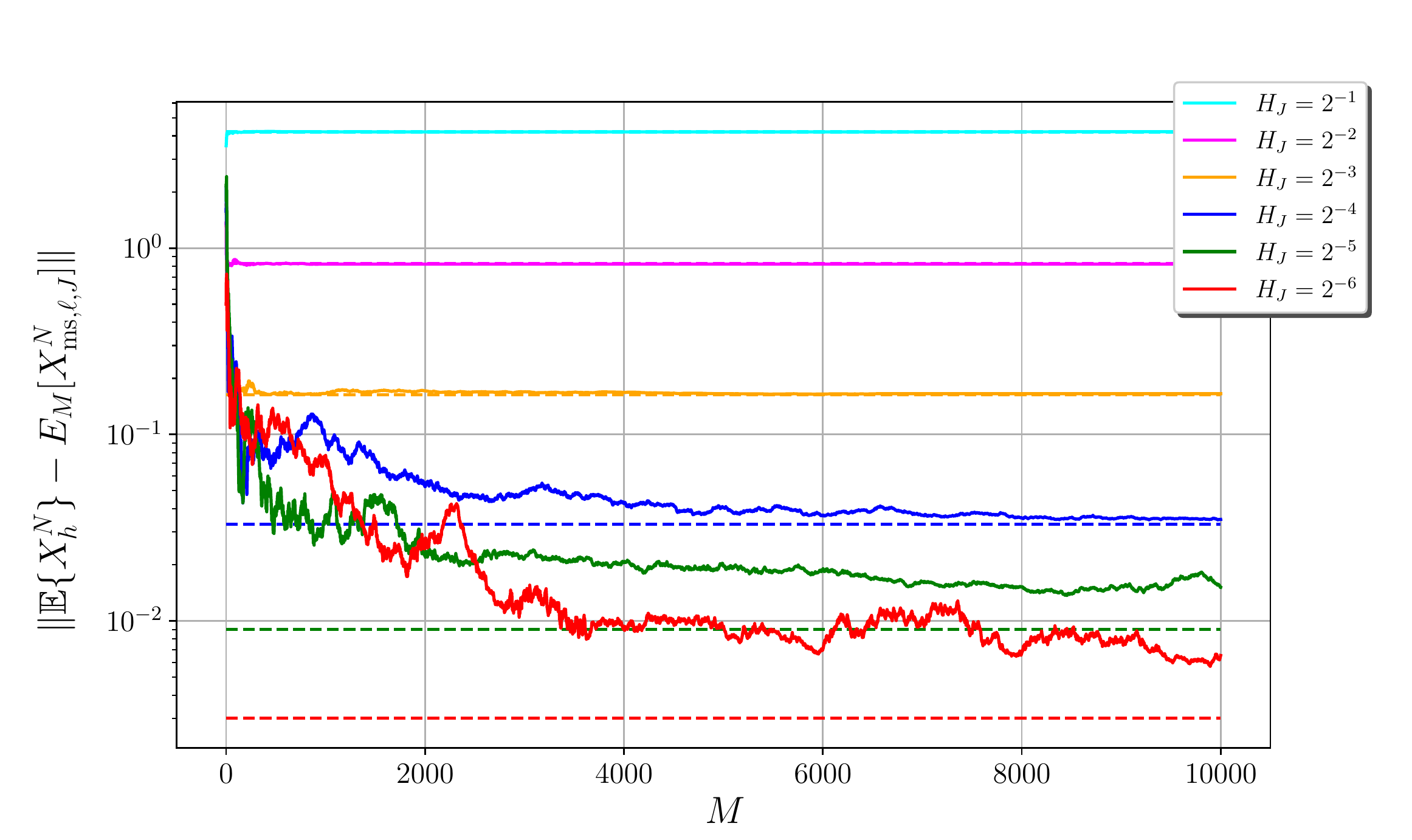}
	\caption{The error for the Monte-Carlo estimator of the LOD approximation as function of the number of samples $M$. The dashed lines indicate the discretization error for each coarse mesh size $H_J$, respectively.\label{fig:weaK_error_M}}
\end{figure}

\begin{figure}
	\centering
	\begin{tikzpicture}
	\begin{axis}[
	scale=1,
	width=4.602in,
	height=3.106in,
	legend cell align={left},
	legend style={
		fill opacity=0.8,
		draw opacity=1,
		text opacity=1,
		at={(0.97,0.03)},
		anchor=south east,
		draw=lightgray204
	},
	log basis x={2},
	log basis y={2},
	tick align=outside,
	tick pos=left,
	x grid style={darkgray176},
	xlabel={\(\displaystyle H_J\)},
	xmajorgrids,
	xmin=0.0131390064883393, xmax=0.594603557501361,
	xmode=log,
	xtick style={color=black},
	y grid style={darkgray176},
	ylabel={\Large Weak error},
	ymajorgrids,
	ymin=0.00122190562660053, ymax=8.94618720660721,
	ymode=log,
	ytick style={color=black}
	]
	\addplot [semithick, darkorange25512714, mark=triangle*, mark size=3, mark options={solid,rotate=270}]
	table {%
		0.5 3.8
		0.25 0.95
		0.125 0.26
		0.0625 0.0595
		0.03125 0.01498
		0.015625 0.003
	};
	\addlegendentry{LOD-MLMC}
	\addplot [semithick, steelblue31119180, mark=x, mark size=3, mark options={solid}]
	table {%
		0.5 3.8
		0.25 1
		0.125 0.202
		0.0625 0.0886
		0.03125 0.0151
	};
	\addlegendentry{LOD-MC}
	\addplot [semithick, forestgreen4416044, mark=*, mark size=3, mark options={solid}]
	table {%
		0.5 5.97
		0.25 3.45
		0.125 2.84
		0.0625 2.63
		0.03125 2.29
	};
	\addlegendentry{FEM-MC ($H_J$)}
	\addplot [semithick, mediumpurple148103189, opacity=0.4, dashed]
	table {%
		0.5 1.25
		0.25 0.625
		0.125 0.3125
		0.0625 0.15625
		0.03125 0.078125
		0.015625 0.0390625
	};
	\addlegendentry{$\mathcal{O}(H_J)$}
	\addplot [semithick, sienna1408675, opacity=0.4, dashed]
	table {%
		0.5 1.875
		0.25 0.46875
		0.125 0.1171875
		0.0625 0.029296875
		0.03125 0.00732421875
		0.015625 0.0018310546875
	};
	\addlegendentry{$\mathcal{O}(H_J^2)$}
	\end{axis}
	
	\end{tikzpicture}
	\caption{The weak error as function of the coarse mesh size $H_J$. The number of samples are set to $M_J = \gamma H_J^{-4}$ with $\gamma=0.01$ for all Monte-Carlo estimators. For the multilevel estimator, the samples are set to $\Mtilde_0\ = \gamma H_J^{-4}$ and $\Mtilde_j\ =\ \Mtilde_0H_j^4\cdot 2^{2\delta j}$ for $j=1,\ldots,J$ with $\delta=1$.\label{fig:weak_error_H}}
	\vspace{1cm}
	\begin{tikzpicture}
	
	\begin{axis}[
	width=4.602in,
	height=3.106in,
	legend cell align={left},
	legend style={fill opacity=0.8, draw opacity=1, text opacity=1, draw=lightgray204},
	log basis x={2},
	log basis y={10},
	tick align=outside,
	tick pos=left,
	x grid style={darkgray176},
	xlabel={\(\displaystyle H_J\)},
	xmajorgrids,
	xmin=0.0131390064883393, xmax=0.594603557501361,
	xmode=log,
	xtick style={color=black},
	y grid style={darkgray176},
	ylabel={\Large Time [s]},
	ymajorgrids,
	ymin=1.15418785140679, ymax=27981675.5657545,
	ymode=log,
	ytick style={color=black}
	]
	\addplot [semithick, finefem, mark=diamond*, mark size=3, mark options={solid}]
	table {%
		0.5 77
		0.25 231
		0.125 3157
		0.0625 50512
		0.03125 807422
		0.015625 12918444
	};
	\addlegendentry{FEM-MC $(h)$}
	\addplot [semithick, steelblue31119180, mark=x, mark size=3, mark options={solid}]
	table {%
		0.5 2.5
		0.25 8.1
		0.125 33
		0.0625 278
		0.03125 5548
		0.015625 1212167
	};
	\addlegendentry{LOD-MC}
	\addplot [semithick, darkorange25512714, mark=triangle*, mark size=3, mark options={solid,rotate=270}]
	table {%
		0.5 2.5
		0.25 10.6
		0.125 43
		0.0625 188
		0.03125 1468
		0.015625 20380
	};
	\addlegendentry{LOD-MLMC}
	\end{axis}
	
	\end{tikzpicture}
	
	\caption{The running time for the different approaches of computing an estimation of the quantity $\E{X^N_{\mathrm{ms}, \ell, J}}$. The samples for the FEM-MC and LOD-MC estimators are chosen as $M_J = \gamma H_J^{-4}$, and for the LOD-MLMC estimator the samples are chosen as $\Mtilde_0\ = \gamma H_J^{-4}$ and $\Mtilde_j\ =\ \Mtilde_0H_j^4\cdot 2^{2\delta j}$, with parameters $\gamma=0.01$ and $\delta=1$. \label{fig:cost}}
\end{figure}

In Section~\ref{s:weak_error}, it was shown how the Monte-Carlo error consists of the weak error and a statistical part, respectively. As an example, we compute $E_{M}[X^N_{\mathrm{ms},\ell,J}]$ for $H_J = 2^{-(J+1)}$, $J=0,\ldots,5$, with $M=10000$ samples on each mesh size, and show how the total error decays as a function of $M$. This is illustrated in Figure~\ref{fig:weaK_error_M}, where we furthermore have included dashed horizontal lines that indicate the size of the deterministic error for each coarse mesh size. Note that after 10000 samples the statistical error has vanished for $H_J=2^{-1},\ldots,2^{-4}$, while it is still present for $H_J=2^{-5},2^{-6}$. As pointed out in Section~\ref{s:weak_error}, the weak and statistical errors are balanced by choosing the samples of the Monte-Carlo estimator proportional to the mesh size as $M_J \sim H_J^{-4}$. 

Next, we compute an estimation of $\E{X^N_{\mathrm{ms},\ell,J}}$ using standard Monte-Carlo estimation and multilevel Monte-Carlo estimation, respectively, where the number of samples is chosen as a function of the coarse mesh size $H_J$. In the following, we set $\gamma=0.01$ as a scaling parameter for the number of samples. For the standard Monte-Carlo estimator $E_{M_J}[X^n_{\mathrm{ms},\ell,J}]$ we use $M_J = \gamma H_J^{-4}$ number of samples. This is done for coarse mesh sizes $H_J = 2^{-(J+1)}$, $J=0,\ldots,4$. For the multilevel Monte-Carlo estimator $E^J[X^n_{\mathrm{ms},\ell,J}]$ we set $\Mtilde_0\ = \gamma H_J^{-4}$ and let $\Mtilde_j\ =\ \Mtilde_0H_j^4\cdot 2^{2\delta j}$, for $j=1,\ldots, J$, with $\delta = 1$. This is done for mesh sizes $H_J =  2^{-(J+1)}$, $J=0,\ldots,5$. The convergence of the weak error for both estimators is illustrated in Figure~\ref{fig:weak_error_H}. For comparison, we have furthermore included the Monte-Carlo estimation of the finite element solution based on the coarse grid, i.e., $E_{M_J}[X^N_{H_J}]$ with $M_J = \gamma H_J^{-4}$. 

In the figure, we first of all note that both the Monte-Carlo and multilevel Monte-Carlo estimator of the LOD solution converge with quadratic rate, as predicted in Section~\ref{s:weak_error}. It is moreover shown that the finite element solution on the coarse grid is unable to reach the region of convergence as a consequence of not resolving the variations in the diffusion.

\subsection{Computational complexity}
\label{s:computational_complexity}

We conclude with a note on the computational complexity of computing $\E{X^N_{\mathrm{ms},\ell, J}}$ using the Monte-Carlo estimator and the multilevel estimator for the LOD method, and compare it with the standard Monte-Carlo estimation of the finite element method on the fine grid. As an example, we wish to compute the solution on an under-resolved grid with mesh size $H_J$ and let $h$ be sufficiently fine, i.e., it resolves the variations on the scale $\varepsilon$.

To simplify the presentation, we use the following notations. By $\mathcal{W}_{H_j}$ and $\mathcal{W}_h$ we denote the computational cost of solving a sparse matrix system defined by mesh sizes $H_j$ and $h$, respectively. Furthermore, $\mathcal{W}_{h,K}^\ell$ denotes the cost of computing a matrix system on the fine scale $h$ localized to the patch $N^\ell(K)$. Recall that the construction of the multiscale space $\Vmsl$ consists of evaluating the basis correction $\Rfl \phi_i$ for every node $x_i \in \NH$. In practice, this is done by solving $2^d-1$ number of fine (localized) system for each element $K\in \KH$.  That is, the computational cost of evaluating the multiscale space $\Vmsl$ based on the coarse mesh size $H_j$ is $(2^d-1)\cdot |\mathcal{K}_{H_j}| \cdot \mathcal{W}_{h,K}^\ell =: \mathcal{V}^\ell_j$ (see \cite{EngHMP19} for a more detailed discussion). However, we emphasize that this cost can be seen as an offline computation, and becomes of negligible size as the number of time steps $N$ and samples $M$ grow.

\begin{enumerate}[start=1, wide = 0pt, leftmargin = 0em, parsep=0.5em]
	\item[\bf FEM-MC:] For the Monte-Carlo estimation, we require $M_J \sim H^{-4}_J$ number of samples of $X^N_h$. For each sample, one must solve the system~\eqref{eq:fully_discrete_galerkin} $N$ number of times. In total, it is necessary to solve $M_J\cdot N$ matrix systems defined on the fine grid with mesh size $h$. The computational cost is thus $M_J\cdot N \cdot \mathcal{W}_h$.
	\item[\bf LOD-MC:] The Monte-Carlo estimator requires $M_J\sim H^{-4}_J$ number of samples $X^N_{\mathrm{ms},\ell,J}$. To compute the samples, it is first necessary to construct the multiscale space $\Vmsl$. We emphasize that once this is done, the multiscale space $\Vmsl$ can be re-used in each time step and for each sample. In total, we are required to solve for $\Vmsl$ once, followed by solving $M_J\cdot N$ matrix systems posed on $\Vmsl$ defined by the coarse mesh size $H_J$. The total cost is thus $\mathcal{V}^\ell_J + M_J\cdot N \cdot \mathcal{W}_{H_J}$.
	\item[\bf LOD-MLMC:] For the multilevel estimator, we are first required to construct the multiscale spaces $\Vmsl$ for the coarse mesh sizes $H_j$, $j=0,\ldots,J$. Next, the estimator $E_{\Mtildelow_0}[X^N_{\mathrm{ms},\ell,0}]$ is computed, which requires $\Mtilde_0\ \sim H_J^{-4}$ samples of $X^N_{\mathrm{ms},\ell,0}$. At last, the estimators $E_{\Mtildelow_j}[X^N_{\mathrm{ms},\ell,j} - X^N_{\mathrm{ms},\ell,j-1}]$ for $j=1,\ldots,J$ are evaluated, which is done by computing $\Mtilde_j\ \sim H_J^{-4}H_j^{4}\cdot 2^{2\delta j}$ samples of $X^N_{\mathrm{ms},\ell,j}$ and $X^N_{\mathrm{ms},\ell,j-1}$, respectively. Recall that one sample of $X^N_{\mathrm{ms},\ell,j}$ corresponds to $N$ sparse matrix systems defined by the grid size $H_j$. In total, this adds to the computational cost 
	\begin{align}
		\sum_{j=0}^{J} \mathcal{V}^\ell_j + (\Mtilde_j + \Mtilde_{j+1}) \cdot N \cdot \mathcal{W}_{H_j},
	\end{align}
	where we have set $\Mtilde_{J+1}\ =0$. 
\end{enumerate}

Note that, although the FEM-MC and LOD-MC methods require the same number of samples $M_J\sim H_J^{-4}$, the LOD-MC computes the samples on the coarse scale $H_J$ while FEM-MC does so on the fine scale $h$. We emphasize that the LOD-MLMC estimation is even more advantageous, since the number of samples computed on the scale $H_J$ is always $\Mtilde_J \sim 2^{2\delta J}$, i.e., independent of the total number of samples. To illustrate the difference, the computational time for the simulations from the previous subsection were logged. This was done by taking the average computational time for one sample for each method, and multiplying it by the total number of samples. That is, for FEM-MC and LOD-MC, the number of samples are $M_J = \gamma H_J^{-4}$, and for the LOD-MLMC estimation we have $\Mtilde_0\ = \gamma H_J^{-4}$ and $\Mtilde_j\ =\ \Mtilde_0H_j^4\cdot 2^{2\delta j}$, with $\gamma= 0.01$ and $\delta= 1$. The computational times were computed for coarse mesh sizes $H_J=2^{-1},\ldots,2^{-6}$ and can be seen in Figure~\ref{fig:cost}. Note that, for instance, the LOD-MC was never computed in the case $H_J=2^{-6}$ in the previous subsection, due to computational limitations. However, since the computational time for one sample is known, it is possible to estimate the total time that is required for the entire estimation. In the figure, we note that the computational time for FEM-MC is magnitudes larger than each LOD-based method. The contrast between LOD-MC and LOD-MLMC is small for coarse mesh sizes, but as $H_J$ decreases the LOD-MLMC method shows significant advantage.

\bibliographystyle{abbrv}
\bibliography{References}

\end{document}